\documentclass
[11pt,a4paper,oneside,english,english]
{article}
\usepackage[latin1]{inputenc}
\usepackage{indentfirst}

\usepackage{amsmath,amsfonts,amssymb,amsthm,amscd}
\RequirePackage{ifthen}
\RequirePackage{hyperref}
\usepackage{latexsym}
\usepackage{mathrsfs}
\usepackage{algorithm}
\usepackage[noend]{algorithmic}
\usepackage{paralist}
\usepackage{graphicx}
\usepackage{subfig}
\usepackage{booktabs}
\input xy  \xyoption{all}

\usepackage[a4paper,top=2.0cm,bottom=2.0cm,left=1.9cm,right=1.9cm]{geometry}

\theoremstyle{plain}

\theoremstyle
{plain}
\newtheorem{teorema}{Theorem}[section]
\newtheorem{theorem}[teorema]{Theorem}
\newtheorem{proposizione}[teorema]{Proposition}
\newtheorem{proposition}[teorema]{Proposition}
\newtheorem{lemma}[teorema]{Lemma}
\newtheorem{corollario}[teorema]{Corollary}
\newtheorem{fact}[teorema]{Fact}
\newtheorem{domanda}[teorema]{Question}
\newtheorem{claim}[teorema]{Claim}
\newtheorem{problem}[teorema]{Problem}
\theoremstyle{definition}
\newtheorem{definizione}[teorema]{Definition}
\newtheorem{esempio}[teorema]{Example}
\newtheorem{example}[teorema]{Example}
\newtheorem{osservazione}[teorema]{Remark}
\newtheorem{remark}[teorema]{Remark}

\makeindex

\newcommand{\N}{\mathbb{N}}
\newcommand{\Z}{\mathbb{Z}}
\newcommand{\Q}{\mathbb{Q}}
\newcommand{\R}{\mathbb{R}}

\newcommand{\T}{\mathbb{T}}

\newcommand{\B}{\mathscr B}
\newcommand{\D}{\mathcal D}
\newcommand{\Sm}{\mathcal S}
\newcommand{\BB}{\mathcal B}

\newcommand{\E}{\mathcal E}

\newcommand{\II}{\mathcal I}

\newcommand{\NB}{$\clubsuit\ $}

\newcommand{\XX}{\mathcal X}
\newcommand{\YY}{\mathcal Y}
\newcommand{\cl}{\mathfrak{cl}}
\newcommand{\rC}{\mathfrak{r}\mathcal{C} }

\def\CSsim{\mathbf{Coarse/\!_\sim}}
\def\CS{\mathbf{Coarse}}
\def\CGQsim{\mathbf{CGrpQ/\!_\sim}}
\def\CGsim{\mathbf{CGrp/\!_\sim}}
\def\lCG{\mathbf{l\mbox{-}CGrp}}
\def\rCG{\mathbf{r\mbox{-}CGrp}}
\def\CGQ{\mathbf{CGrpQ}}
\def\CG{\mathbf{CGrp}}

\def\kappaCGQ{\mathbf{\kappa\mbox{-}CGrpQ}}
\def\kappaCG{\mathbf{\kappa\mbox{-}CGrp}}
\def\lCGQ{\mathbf{l\mbox{-}CGrpQ}}
\def\rCGQ{\mathbf{r\mbox{-}CGrpQ}}
\def\lkappaCG{\mathbf{l\mbox{-}\kappa\mbox{-}CGrp}}
\def\rkappaCG{\mathbf{r\mbox{-}\kappa\mbox{-}CGrp}}

\def\kappaCGQsim{\mathbf{\kappa\mbox{-}CGrpQ/\!_\sim}}
\def\kappaCGsim{\mathbf{\kappa\mbox{-}CGrp/\!_\sim}}

\def\perf{perfect}

\DeclareMathOperator{\asdim}{asdim}

\DeclareMathOperator{\Tor}{Tor}

\DeclareMathOperator{\n}{n}
\DeclareMathOperator{\Mor}{Mor}
\DeclareMathOperator{\Ufun}{U}

\DeclareMathOperator{\Ffun}{F}
\DeclareMathOperator{\Gfun}{G}

\DeclareMathOperator{\Qfun}{Q}
\DeclareMathOperator{\Funct}{Funct}

\author{Dikran Dikranjan\footnote{This author was partially supported by grant PSD-2015-2017-DIMA - progetto PRID TokaDyMA
 of Udine University.}
\  and \  Nicol\`o Zava\footnote{The second author was partially supported by the Project SIR2014 - GADYGR.}
\\  \\ {\footnotesize Department of Mathematics, Computer Science and Physics, University of Udine}
\\ {\footnotesize Via Delle Scienze 206, 33100 Udine, Italy}\\{\footnotesize \tt dikran.dikranjan@uniud.it}, \  {\footnotesize \tt nicolo.zava@gmail.com}
}

\title{Categories of coarse groups: quasi-homomorphisms and functorial coarse structures 
\thanks{MSC: 18B30, 18B99, 18E35, 20K45, 54E99, 54H11.
\endgraf{Keywords}: quasi-homomorphism, functorial coarse structure, coarse group, group ideal, localization of a category, coarse space, coarse equivalence, cardinal invariants.}
}
\date{}
\begin{document}
\maketitle
\begin{abstract}
Coarse geometry is the study of large-scale properties of spaces. In this paper we study group coarse structures (i.e., coarse structures on groups that agree with the algebraic structures), by using group ideals. We introduce a large class of examples of group coarse structures induced by cardinal invariants. In order to enhance the categorical treatment of the subject, we use quasi-homomorphisms, as a large-scale counterpart of homomorphisms. In particular, the localisation of a category plays a fundamental role. We then define the notion of functorial coarse structures and we give various examples of those structures. 
\end{abstract}

%


\section*{Introduction}



Coarse geometry, also known as large-scale geometry is the study of large-scale properties of spaces, ignoring their local, small-scale ones. This theory found applications in several branches of mathematics, for example in geometric group theory (following the work of Gromov on finitely generated groups endowed with their word metrics), in Novikov conjecture, and in coarse Baum-Connes conjecture. We refer to \cite{NowYu} for a comprehensive introduction to large-scale geometry of metric spaces, and to \cite{Har} for applications to geometric group theory.

Large-scale geometry was initially developed for metric spaces, but then several equivalent structures that capture the large-scale properties of spaces appeared, inspired by the theory of uniform spaces (\cite{Isb}). Roe introduced coarse spaces (\cite{roe}), as a counterpart of Weil's definition of uniform spaces via entourages, Protasov and Banakh (\cite{ProBa}) defined balleans, generalising the ball structure of metric spaces, Dydak and Hoffland with large-scale structures (\cite{DydHof}) and Protasov with asymptotic proximities (\cite{Proprox}) independently developed the approach via coverings, as Tukey did for uniform spaces. 
 As for the definition of coarse structures and coarse spaces, we refer to Definition \ref{Def:Roe}. In \cite{DikZa_cat} the equivalence between those structures is deeply studied and a categorical treatment of the subject is provided. In particular, the category $\CS$, of coarse spaces and bornologous maps (Definition \ref{def:Mor}) between them is considered, as well as its quotient category $\CSsim$, where $\sim$ is the closeness relation between morphisms. Some properties of both categories are shown, for example, $\CS$ is a topological category. 
Moreover, we characterised in \cite{DikZa_cat} both the monomorphisms and the epimorphisms of $\CSsim$, showing that it is a balanced category. In \cite{Zacwp} this study is pushed further, establishing, among others, cowellpoweredness of $\CSsim$.

In this paper we are interested in coarse structures on groups, aiming for a categorical treatment of the subject. We require that those coarse structures agree with the algebraic structures of the supporting group and this idea leads to the definition (Definition \ref{def:ideal}) of left (right) group coarse structures (and thus to left and right coarse groups). If a coarse structure on a group is both a left and 
right group coarse structure, we say that it is a uniformly invariant group coarse structure and the coarse group is called a bilateral coarse group. If there is no risk of ambiguity, for the sake of simplicity, we will refer to left group coarse structures as group coarse structures. The study of coarse groups was started by Protasov in \cite{Sketchgroup}, where he introduced this notion by using balleans. In the same paper he highlighted the fact that coarse groups are uniquely determined by a particular ideal of subsets of the group, called group ideal (Definition \ref{def:group_ideal}). The idea is similar to the fact that every group topology is uniquely determined by the filter of neighbourhoods of the identity. More recently, Nicas and Rosenthal (\cite{nicasros2}) developed the same approach via entourages.

We aim to define categories of coarse groups. The first choice is $\lCG$, whose objects are (left) coarse groups and whose morphisms are bornologous homomorphisms. Taking the quotient category $\lCG/_\sim$ of $\lCG$ under the closeness relation would be the next step. However we face some undesired consequences even dealing with basic examples. In fact, the inclusion homomorphism $i\colon \Z\to\R$, where both groups are endowed with the usual euclidean metric, is one of the first examples of coarse equivalences (see Definition \ref{def:Mor}). However, there is no coarse inverse of $i$ which is a homomorphism. Hence $[i]_\sim$ is not an isomorphism of $\lCG/_\sim$. In order to overcome this problem we need the notion of quasi-homomorphism.

A {\em quasi-homomorphism} (also called {\em quasi-morphism}) is a map $f\colon G\to\R$ from a group into the real line which is somehow ``close'' to be a homomorphism, i.e. there exists a constant $K>0$ such that $\lvert f(x)+f(y)-f(xy)\rvert<K$, for every $x,y\in G$. The notion of quasi-homomorphism dates back to some questions posed by Ulam (\cite{Ul}) in the realm of linear functional equations. We refer to \cite{Ul}, \cite{kot} and \cite{kot1} for an introduction to this classical subject 

Rosendal \cite{ros} noticed that the classical notion of quasi-homomorphism can be described and extended to other settings using the large-scale notion of closeness (see Definition \ref{def:q_hom} for a rigorous definition). Also in Fujiwara and Kapovich \cite{FujKap}, who followed some older sources, there is a generalisation of the classical notion of quasi-homomorphism. 

In this paper we study quasi-homomorphisms, in Rosendal's definition, in order to refute Kotschick's point of view: ``the notion of a quasi-morphism does not have much to do with category theory'' (\cite{kot}). We prove that, in the class of bilateral coarse groups (that properly contains all abelian coarse groups), maps close to quasi-homomorphisms are quasi-homomorphisms (Proposition \ref{prop:close_quasi_homo}) and composites of bornologous quasi-homomorphisms are bornologous quasi-homomorphisms (Proposition \ref{prop:comp_quasi-homo}). Finally, in the same class of coarse groups, we show that coarse inverses of quasi-homomorphisms that are coarse equivalences are quasi-homomorphisms. In particular, every coarse inverse of the inclusion map $i\colon\Z\to\R$ is a quasi-homomorphism (for example, the floor map $\lfloor\cdot\rfloor$).

We then define the quotient category $\CGQsim$ of bilateral coarse groups and equivalence classes of bornologous quasi-homomorphisms between them. In this category, every equivalence class of a homomorphism which is a coarse equivalence is an isomorphism. 

In \S\ref{sub:loc} we study the localisation $\CGsim[\mathcal W^{-1}]$ of the quotient category $\CGsim$, of bilateral coarse groups and equivalence classes of bornologous homomorphisms, by the family $\mathcal W$ of equivalence classes of homomorphisms which are coarse equivalences. The category $\CGsim[\mathcal W^{-1}]$, provided that it exists, is the ``smallest'' category containing $\CGsim$ for which all morphisms of $\mathcal W$ are isomorphisms. We then ask whether it exists and if it coincides with $\CGQsim$. As for the existence, we provide in Corollary \ref{coro:fin_loc} a positive answer in the case of $\kappa$-group coarse structures (in particular for the finitary one), for which a nice characterisation of morphisms is provided.

The group coarse structures used in \cite{nicasros2} and \cite{ros} are examples of {\em functorial coarse structures} (see Example \ref{ex:group_c_s} for detail). This should be compared with the notion of functoriality, appearing in the category of topological groups as follows. 
A {\em functorial topology} is a functor $\Ffun\colon {\bf Grp} \to {\bf TopGrp}$ that assigns to 
every {\em abstract} group $G$ a group topology $T_G$ so that $F\colon G \mapsto (G,T_G)$ is a 
functor $F\colon {\bf Grp} \to {\bf TopGrp}$, i.e., every group homomorphism $f\colon G \to H$ in {\bf Grp} gives rise to a continuous group homomorphism $f\colon (G,T_G) \to( H, T_H)$ in {\bf TopGrp} (\cite{fu,ma}).
%
Inspired by the existing examples of coarse structures on (topological) groups given in \cite{nicasros2}, we define a functorial coarse structure of groups as a functor $\Ffun\colon {\bf Grp} \to\XX$, where $\XX$ is $\lCG$ or $\rCG$, such that 
$F(G)=(G,\E)$,  i.e., $\Ufun\circ\Ffun$ is the identity functor,  
where $\Ufun$ is the forgetful functor $\Ufun\colon{\CG}\to {\bf Grp}$, defined by $\Ufun(G,\E)=G$ and similarly on morphisms. 

In \S \ref{section_linear_coarse} we introduce coarse structures induced by cardinal invariants
using ideals generated by subgroups (linear coarse structures). They are defined in \S \ref{LinStr}. 

In \S \ref{Sec:last} we scrutinise abelian groups under the looking glass of the functorial coarse structure induced
by the free-rank, establishing a kind of ``rigidity" of the class of divisible groups with respect to homomorphisms that are 
 coarse equivalences. 
For example, in Theorem \ref{ex:example} we prove that 
if a  fully decomposable torsion-free abelian group $G$ is coarsely equivalent (i.e., ``as close as possible'' from the large-scale point of view) to a divisible group, then $G$ is also ``as close as possible'' to a divisible group from algebraic point of view (i.e., $r_0(G/d(G)) < \omega$), in case $G$ is either  uncountable  or homogeneous.  
 These results go  close, more or less, to the spirit of the nice results obtained by Banakh, Higes and Zarichnyi \cite{banhzar} 
 where the asymptotic dimension was used to this end.

The paper is organised as follows. In Section \ref{sec:cs_cg} we introduce the objects we focus on, in particular, in \S\ref{sub:cs} we recall the definitions of coarse spaces and morphisms, while in \S\ref{sub:cg} we pass to group coarse structures and coarse groups, giving also the characterisation using group ideals, and we provide the important characterisation of bilateral coarse groups. In the same subsection, we give many examples of group ideals defined both on groups and on topological groups. Then Section \ref{sec:group_ideals} is devoted to show how coarse notions can be rewritten in terms of group ideals for coarse groups. In Section \ref{section_linear_coarse} we introduce and study group ideals (equivalently, group coarse structures) defined by cardinal and numerical invariants. In particular, these structures are introduced in \S\ref{LinStr}, while in \S\ref{Sec:last} we focus on those induced by the free-rank, and in \S\ref{sub:Ban_Che_Lya} we apply those results discussing a problem posed by Banakh, Chervak and Lyaskovska. The notion of quasi-homomorphism is the focus of Section \ref{section:quasi_hom}. Finally, Section \ref{sec:cat_coarse_grps} is dedicated to developing a categorical approach to the theory of coarse groups. We define the categories of coarse groups and bornologous (quasi-)homomorphisms   between them, we introduce functorial coarse structures (\S\ref{sub:funct}) and we prove some technical results concerning pullbacks (\S\ref{sub:pull}) in order to discuss the localisation of the category $\CGsim[\mathcal W^{-1}]$ (\S\ref{sub:loc}).

\subsection*{Notations and terminology}

In the sequel, we adopt the standard notation in group theory following \cite{fu,DPS,rob}.
In particular, we denote by $0$ the identity of an abelian group $G$ and by $\Tor(G)$ its torsion subgroup. Furthermore,  
$\N$, $\Z$, $\Q$ and $\R$ denote the sets of positive integers, of integers, of rational numbers and of real numbers, respectively. If $G$ is an abelian group, and $m\in\N$, $G[m]:=\{x\in G\mid mx=0\}$. If $G$ is an abelian group, $r_0(G)$ denotes the free rank of $G$, i.e., the cardinality of the maximal independent subset of $G$, while, for every prime $p$, $r_p(G)$ denotes the $p$-rank of $G$, i.e., the value $r_p(G):=\dim_{\Z/p\Z}G[p]$, the dimension of $G[p]$ as a linear space over the finite field $\Z/p\Z$.

\section{Coarse spaces and coarse groups}\label{sec:cs_cg}

\subsection{Coarse spaces}\label{sub:cs}

\begin{definizione}\label{Def:Roe} According to Roe (\cite{roe}), a \emph{coarse space} is a pair $(X,\E)$, where $X$ is a set and $\E\subseteq\mathcal P(X\times X)$ a \emph{coarse structure} on it, which means that
\begin{compactenum}[(i)]
  \item $\Delta_X:=\{(x,x)\mid x\in X\}\in\E$;
  \item $\E$ is an {\em ideal}, i.e., it is closed under taking finite unions and subsets;
  \item if $E\in\E$, then $E^{-1}:=\{(y,x)\in X\times X\mid(x,y)\in E\}\in\E$;
  \item if $E,F\in\E$, then $E\circ F:=\{(x,y)\in X\times X\mid\exists z\in X:\,(x,z)\in E,(z,y)\in F\}\in\E$. 
\end{compactenum}
 An element $E$ of $\E$ is called {\em entourage}. We say that an entourage $E$ is {\em symmetric} if $E^{-1}=E$.
\end{definizione}

If $X$ is a set, a {\em base of a coarse structure} is a family $\BB$ of entourages such that its {\em completion} $\cl(\BB):=\{F\subseteq B\mid B\in\BB\}$ is a coarse structure. Note that $\cl(\BB)$ is the closure of $\BB$ under taking subsets. For example, for every coarse structure $\E$, the family of all symmetric entourages forms a base of $\E$.

 If $(X,\E)$ is a coarse space and $Y$ is a subset of $X$, then $Y$ can be endowed with the {\em subspace coarse structure} $\E|_Y:=\{E\cap(Y\times Y)\mid E\in\E\}$.

If $(X,\E)$ is a coarse space and $x\in X$, a subset $B$ of $X$ is {\em bounded from $x$} if there exists an entourage $E\in\E$ such that $B\subseteq E[x]:=\{y\in X\mid (x,y)\in E\}$. A subset is {\em bounded} if it is bounded from a point. A subset $L\subseteq X$ is {\em large in $X$} if there exists $E\in\E$ such that $E[L]:=\bigcup_{x\in L}E[x]=X$.

 Let $(X,\E)$ be a coarse space and $x\in X$ be a point. The {\em connected component of $x$} is the subset $\mathcal Q_X(x):=\bigcup_{E\in\E}E[x]$. A coarse space $(X,\E)$ is {\em connected} if, for every $x\in X$, $X=\mathcal Q_X(x)$ or, equivalently, if $\bigcup\E=X\times X$.

We say that two maps $f,g\colon S\to(X,\E)$ from a set to a coarse space are \emph{close}, and we write $f\sim g$, if $\{(f(x),g(x))\mid x\in S\}\in\E$.
\begin{definizione}\label{def:Mor}
Let $(X,\E_X)$ and $(Y,\E_Y)$ be two coarse spaces. A map $f\colon X\to Y$ is
\begin{compactenum}[(i)]
  \item \emph{bornologous} if $(f\times f)(E)\in\E_Y$ for all $E\in\E_X$;
  \item {\em large-scale injective} if $(f\times f)^{-1}(\Delta_Y)\in\E_X$;
  \item {\em weakly uniformly bounded copreserving} (\cite{Za}) if, for every $E\in\E_Y$, there exists $F\in\E_X$ such that $E\cap(f(X)\times f(X))\subseteq(f\times f)(F)$;
  \item {\em uniformly bounded copreserving} (\cite{Za}) if, for every $E\in\E_Y$, there exists $F\in\E_X$ such that $F[f(x)]\cap f(X)\subseteq f(E[x])$, for every $x\in X$;
  \item {\em proper} if $f^{-1}(B)$ is bounded for every bounded set $B$ of $Y$;
  \item {\em effectively proper} if  $(f\times f)^{-1}(E)\in\E_X$ for all $E\in\E_Y$;
  \item a {\em coarse embedding} if it is both bornologous and effectively proper;
  \item an {\em asymorphism} if $f$ is bijective and both $f$ and $f^{-1}$ are bornologous;
  \item a \emph{coarse equivalence} if $f$ is bornologous and one of the following equivalent conditions holds:
\begin{compactenum}
  \item[(ix,a)] there exists a bornologous map $g\colon Y\to X$, called {\em coarse inverse of $f$} such that $g\circ f\sim id_X$ and $f\circ g\sim id_Y$;
  \item[(ix,b)] $f$ is effectively proper and $f(X)$ is large in $Y$ (i.e., $f$ is {\em large-scale surjective}).
\end{compactenum}
\end{compactenum}
A family of maps $\{f_i\colon X\to Y\}_{i\in I}$ is {\em uniformly bornologous} if, for every $E\in\E_X$, there exists $F\in\E_Y$ such that $(f_i\times f_i)(E)\subseteq F$, for every $i\in I$. 
\end{definizione}

A family $\mathcal U$ of subsets of a coarse space $(X,\E)$ is {\em uniformly bounded} if there exists $E\in\E$ such that, for every $U\in\mathcal U$ and $x\in U$, $U\subseteq E[x]$. With this notion, we can immediately give a different characterisation of large-scale injectivity: a map between coarse spaces is large-scale injective if and only if it has uniformly bounded fibers.

\begin{remark} Let $f\colon X\to Y$ be a map between coarse spaces. Then it canonically factorises as
$$
\xymatrix{
X\ar^f[rr]\ar_{e_f}[dr]&& Y,\\
& f(X)\ar_{m_f}[ur] &
}
$$
where $f(X)$ is endowed with the subspace coarse structure inherited by $Y$, $e_f$ is a surjective bornologous map, and $m_f$ is an asymorphic embedding. Then $f$ is effectively proper (uniformly bounded copreserving, or weakly uniformly bounded copreserving) if and only if $e_f$ is effectively proper (uniformly bounded copreserving, or weakly uniformly bounded copreserving, respectively).
\end{remark}

The following implications between some the previous concepts were pointed out in \cite{Za}: 

\begin{proposition}\label{prop:eff_cop_wcop}
Let $f\colon(X,\E_X)\to(Y,\E_Y)$ be a map between coarse spaces. Then:
\begin{compactenum}[(i)]
	\item if $f$ is effectively proper, then $f$ is uniformly bounded copreserving;
	\item if $f$ is uniformly bounded copreserving, then $f$ is weakly uniformly bounded copreserving.
\end{compactenum}
\end{proposition}

In \cite{Za} it is proved that the previous implications cannot be reverted in general. However, if the map is large-scale injective, then those concepts are equivalent.

\begin{proposition}[\cite{Za}]\label{prop:lsi_eff_pr_ubc_wubc}
Let $f\colon(X,\E_X)\to(Y,\E_Y)$ be a large-scale injective map between coarse spaces. Then the following properties are equivalent:
\begin{compactenum}[(i)]
\item $f$ is effectively proper;
\item $f$ is uniformly bounded copreserving;
\item $f$ is weakly uniformly bounded copreserving.
\end{compactenum}
\end{proposition}

As a consequence, we have another characterisation of coarse equivalences.

\begin{corollario}[\cite{Za}]\label{coro:ls_bijective_ce}
Let $f\colon(X,\E_X)\to(Y,\E_Y)$ be a map between coarse spaces. Then $f$ is a coarse equivalence if and only if the following conditions hold:
\begin{compactenum}[(i)]
\item $f$ is both large-scale injective and large-scale surjective;
\item $f$ is bornologous;
\item $f$ is uniformly bounded copreserving.
\end{compactenum}
\end{corollario}

By virtue of Proposition \ref{prop:lsi_eff_pr_ubc_wubc}, the current condition on $f$ in item (iii) of Corollary \ref{coro:ls_bijective_ce} can be replaced 
by the weaker condition ``weakly uniformly bounded copreserving".

Let $\{(X_i,\E_i)\}_{i\in I}$ be a family of coarse spaces. Let $X=\Pi_iX_i$ and $p_i\colon X\to X_i$, for every $i\in I$ be the projection maps. Then the {\em product coarse structure} $\E=\Pi_i\E_i$ is defined by the base
$$
\mathcal B=\bigg\{\bigcap_{i\in I}(p_i\times p_i)^{-1}(E_i)\mid E_i\in\E_i,\,\forall i\in I\bigg\}.
$$

Let us now introduce a class of coarse spaces. A coarse space $(X,\E)$ is {\em cellular} if, for every $E\in\E$, $E^\square:=\bigcup_{n\in\N}E^n\in\E$, where, for every $n\in\N$, $E^n$ is obtained by composing $E$ with itself $n$ times. Cellular coarse spaces are precisely those with asymptotic dimension $0$ (see \cite{ProZa}).

\subsection{Coarse groups}\label{sub:cg}

In this paper we are interested in coarse structures on (topological) groups that agree with the (topological and) algebraic structure of the spaces.

If $G$ is a group and $g\in G$, we define the {\em left-shift $s_g^\lambda\colon G\to G$} and the {\em right-shift $s_g^\rho\colon G\to G$} as follows: for every $x\in G$, $s_g^\lambda(x)=gx$ and $s_g^\rho(x)=xg$. The following property of left-shifts is easy to check: 

\begin{proposition}\label{prop:compatible*}
	Let $G$ be a group and $\E$ be a coarse structure on it. Then the following properties are equivalent:
  \begin{compactenum}[(i)]
 \item for every $E\in\E$, $GE:=\{(gx,gy)\mid g\in G,\,(x,y)\in E\}\in\E$;
 \item the family $\mathcal S_G^\lambda:=\{s_g^\lambda\mid g\in G\}$ is uniformly bornologous, i.e., for every $E\in\E$, there exists $F\in\E$ such that, for every $g\in G$, $(s_g^\lambda\times s_g^\lambda)(E)\subseteq F$.
  \end{compactenum}       
\end{proposition}

\begin{definizione}
	A coarse structure $\E$ on a group $G$ is said to be a {\em left group coarse structure}, if it has the equivalent properties from Proposition \ref{prop:compatible*}. 
	A {\em left coarse group} is a pair $(G, \E)$ of a group $G$ and a left group coarse structure $\E$ on $G$. {\em Right group coarse structure} and {\em right coarse group} can be defined analogously. 
\end{definizione}

In order to define our leading example of left/right group coarse structures and left/right coarse groups (we shall see below that these are all possible coarse group structures and coarse groups) we need the following fundamental concept. 

\begin{definizione}\label{def:group_ideal}
Let $G$ be a group. A {\em group ideal} $\II$ (\cite{Sketchgroup}) is a family of subsets of $G$ containing the singleton $\{e\}$ such that:
\begin{compactenum}[(i)]
	\item $\II$ is an ideal;
	\item for every $K,J\in\II$, $KJ:=\{kj\mid k\in K,\,j\in J\}\in\II$;
	\item for every $K\in\II$, $K^{-1}:=\{k^{-1}\mid k\in K\}\in\II$.
\end{compactenum}
\end{definizione}
If $\II$ is a group ideal on $G$, $\bigcup\II$ is a subgroup of $G$. 

\begin{definizione}\label{def:ideal}
	Let $G$ be a group and $\II$ be a group ideal. For every $K\in\II$, we define
	$$
	E_K^\lambda:=G(\{e\}\times K):=\bigcup_{g\in G}(\{g\}\times gK).
	$$
	The family $\E_\II^\lambda:=\{E\subseteq G\times G\mid\exists K\in\II:\,E\subseteq E_K^\lambda\}$ is a left coarse group structure, called {\em left $\II$-group coarse structure}, and the pair $(G,\E_{\II}^\lambda)$ is a left coarse group, called {\em left $\II$-coarse group}. 
\end{definizione}
Note that the family $\{E_K^\lambda\mid K\in\II\}$ is a base of the $\II$-group coarse structure. Moreover, for every $K\in\II$ and $x\in G$, $E_K^\lambda[x]=xK$.

Similarly, we can define the {\em right $\II$-group coarse structure} $\E_\II^\rho$ as follows: it is induced by the base $\{E_K^\rho\mid K\in\II\}$, where
$$
E_K^\rho:=(\{e\}\times K)G:=\bigcup_{g\in G}(\{g\}\times Kg).
$$
 Then $(G,\E_{\II}^\rho)$ is called {\em right $\II$-coarse group}.

For every group $G$ and group ideal $\II$ on it, the left $\II$-group coarse structure and the right $\II$-group coarse structure are equivalent, as the following result shows.

\begin{proposition}\label{prop:inv_asym}
	Let $G$ be a group, $\II$ be a group ideal, and $\iota\colon G\to G$ such that $\iota(g)=g^{-1}$. Then $\iota\colon(G,\E_\II^\lambda)\to(G,\E_{\II}^\rho)$ is an asymorphsim. 
\end{proposition}

The following fact from \cite{nicasros2} shows that every left coarse group can be obtained as in Definition \ref{def:ideal} above: 

\begin{proposition}\label{prop:compatible}
	Let $G$ be a group and $\E$ be a coarse structure on it. Then the following properties are equivalent:
   \begin{compactenum}[(i)] 
	\item $(G,\E)$ is a left coarse group; 
	\item $\E=\E_{\II}^\lambda$, where $\II:=\{E[e]\mid E\in\E\}$.
   \end{compactenum}       
\end{proposition}
A similar result can be stated for right coarse structures.

There is another way to describe the group ideal of Proposition \ref{prop:compatible}. If $G$ is a group, the map $\pi_G^\lambda\colon G\times G\to G$ such that, for every $(g,h)\in G\times G$, $\pi_G^\lambda(g,h):=h^{-1}g$ is called ({\em left}) {\em shear map} (\cite{nicasros2}). If $\E$ is a left coarse structure satisfying the properties of Proposition \ref{prop:compatible}, then $\II=\{\pi_G^\lambda(E)\mid E\in\E\}$.

Justified by Propositions \ref{prop:inv_asym} and \ref{prop:compatible}, in the sequel we will always refer to left group coarse structures and left coarse groups, if it is not otherwise stated, and thus we call them briefly group coarse structures (and coarse groups) if there is no risk of ambiguity. 

According to Proposition \ref{prop:compatible}, coarse groups are equivalently described in terms of group ideals. This is why it is necessary to provide examples of group ideals.

\begin{example}\label{ex:group_c_s}  Let $G$ be a group.
 \begin{compactenum}[(i)]
    \item The sigleton $\{\{e\}\}$ is a group ideal and the $\{\{e\}\}$-group coarse structure is the {\em discrete coarse structure}, i.e., the one that contains only the subsets of the diagonal.
    \item On the opposite side we have the group ideal $\mathcal P(G)$, that induces the {\em bounded coarse structure}, i.e., every subset of $G\times G$ is an entourage.
    \item The family $[G]^{<\omega}$ of all finite subsets of $G$ is a group ideal and the $[G]^{<\omega}$-coarse structure is called {\em finitary-group coarse structure}.
   \item We want to generalise the previous example. For a given infinite cardinal $\kappa$, the family $[G]^{<\kappa}:=\{K\subseteq G\mid\lvert K\rvert<\kappa\}$ is a group ideal. The $[G]^{<\kappa}$-group coarse structure is called {\em $\kappa$-group coarse structure}. Then the finitary-group coarse structure is the $\omega$-group coarse structure.
   \item Let $\tau$ be a group topology of $G$. Define $\mathcal C(G)$ as the family of all compact subsets of $G$. Then $\cl(\mathcal C(G))$ coincides with
 the family $\rC(G)$ of all relatively compact subsets of $G$ is a group ideal and the $\rC(G)$-coarse structure is called {\em compact-group coarse structure}. 
   \item Let $d$ be a left-invariant pseudo-metric on $G$. Then the family $\mathcal B_d:=\{A\subseteq G\mid\exists R\ge 0:\,A\subseteq B_d(e,R)\}$ is a group ideal and the $\mathcal B_d$-group coarse structure is called {\em metric-group coarse structure}.
   \item Let $G$ be a topological group. The group ideal 	
 $$
	\mathcal{OB}:=\{A\subseteq G\mid\forall d\text{ left-invariant continuous pseudo-metric, $A$ is $d$-bounded from $e$}\}.
 $$
was defined in \cite{ros}, where other characterisations of $\mathcal{OB}$ are provided. Then
 $$
	\E_L:=\E_{\mathcal{OB}}=\bigcap\{\E_{\BB_d}\mid\text{$d$ is a left-invariant continuous pseudo-metric}\}
 $$
is defined in \cite{ros} and named {\em left-coarse structure}. The group ideal $\mathcal{OB}$ contains the family $\mathcal C(G)$ (and thus $\rC(G)$) and it coincides with $\rC(G)$ if $G$ is locally compact and $\sigma$-compact (\cite[Corollary 2.8]{ros}). However, there exist locally compact groups $G$ with $\rC(G)\subsetneq\mathcal{OB}$. For example, the group $Sym(\N)$ of all permutations of $\N$ endowed with the discrete topology has $\rC(Sym(\N))=[Sym(\N)]^{<\omega}$, while $\mathcal{OB}=\mathcal P(Sym(\N))$ (see \cite[Example 2.16]{ros}).
  \item For an infinite cardinal $\kappa$, a topological space is {\em $\kappa$-Lindel\"of} if every open cover has a subcover of size strictly less than $\kappa$ (so $\omega$-Lindel\"of coincides with compact, while $\omega_1$-Lindel\"of is the standard Lindel\"of property). For a topological group $G$, denote by $\kappa\mbox{-}\mathcal L(G)$ the family of all $\kappa$-Lindel\"of subsets of $G$. Then $\cl(\kappa\mbox{-}\mathcal L(G))$ is a group ideal and the $\cl(\kappa\mbox{-}\mathcal L(G))$-group coarse structure is called {\em $\kappa$-Lindel\"of-group coarse structure}.
 \end{compactenum} 
\end{example}

Note that, if $G$ is a discrete group, then, for every infinite cardinal $\kappa$, the $\kappa$-Lindel\"of-group coarse structure coincides with the $\kappa$-group coarse structure. In particular, the compact-group coarse structure coincides with the finitary-group coarse structure.

According to Proposition \ref{prop:inv_asym}, for every group $G$ and group ideal $\II$, $(G,\E_{\II}^\lambda)$ and $(G,\E_{\II}^\rho)$ are asymorphic. However, 
these two group coarse structures need not coincide in general. It will be useful in the sequel to characterise those group ideals for which these two group coarse structures coincide.

\begin{proposizione}\label{prop:coarseuniformmultiplication}
Let $G$ be a group and $\E$ a coarse structure on it. If the group operation $\mu\colon G\times G\to G$ is bornologous, then $\E$ is both a left and a right group coarse structure.
\end{proposizione}

\begin{proof} For every $E\in\E$, $GE=(\mu\times\mu)(\Delta_X\times E)\in\E$ and $EG=(\mu\times\mu)(E\times\Delta_X)$, and thus the claim follows from Proposition \ref{prop:compatible*}. 
\end{proof}

If $K$ is a subset of  a group $G$, and $g\in G$, we define $K^g:=g^{-1}Kg$ and $K^G:=\bigcup_{h\in G}K^h$. A group ideal $\II$ is {\em uniformly bilateral} if $K^G\in\II$ for every $K\in\II$. 
Note that, for every $K\subseteq G$ and $g\in G$, 

\begin{equation}\label{eq:unif_bilateral}
Kg=gg^{-1}Kg=gK^{g}\subseteq gK^G. 
\end{equation}

Similarly, if $E\subseteq G\times G$, and $g\in G$ be an element,  we define $E^g:=\{(g^{-1}xg,g^{-1}yg)\mid(x,y)\in E\}$ and $E^G:=\bigcup_{h\in G}E^h$. We say that a coarse structure $\E$ on $G$ is {\em uniformly invariant} if $E^G\in\E$ for every $E\in\E$. 

The following proposition is the analogue in realm of coarse groups of \cite[Proposition 1.2]{HerPro}.

\begin{proposizione}\label{prop:coarseuniformoperation}
	Let $G$ be a group and $\E$ is a left $\II$-group coarse structure on it, for some group ideal $\II$ on $G$. Then the following conditions are equivalent:
	\begin{compactenum}[(i)]
		\item the inverse map $i\colon (G,\E)\to (G,\E)$ is bornologous;
		\item the multiplication map $\mu\colon(G\times G,\E\times\E)\to(G,\E)$ is bornologous;
		\item $\E$ is also a right $\II$-group coarse structure;
		\item $\E$ is uniformly invariant;
		\item $\II$ is uniformly bilateral.
	\end{compactenum}
	In particular, when the above conditions are fulfilled, the subgroup $\bigcup\II$ is normal.
\end{proposizione}
A coarse group with uniformly invariant group coarse structure will be called {\em bilateral coarse group}.

In particular, for every abelian group and every group ideal on it, the conditions of Proposition \ref{prop:coarseuniformoperation} are satisfied. 
It is natural to expect that this remains true for groups close to be abelian, e.g., for groups $G$ having large centre with respect to the  finitary-group coarse structure
of $G$. This means that $Z(G)$ has finite index in $G$. Then, by Shur's Theorem \cite{rob}, the commutator subgroup $G'$ is finite. As we shall below, this implies that $\E_\II^\lambda = \E_\II^\rho$
(see Corollary \ref{quot}). Since finiteness of $G'$ can still be considered as a rather strong restraint, we consider now a weaker condition 
(but it ensures uniform invariance only of some group coarse structures). Recall that a group $G$ is called an {\em $FC$-group}, if all conjugacy classes $x^G$ are finite. 
Obviously, $G$ is an $FC$-group, if $G'$ is finite. 

The next proposition shows that this commutativity condition is the precise measure ensuring uniform invariance of the  finitary-group coarse structure. 
Its easy proof will be omitted. 

\begin{proposizione}\label{criterion}
For every group $G$ the following conditions are equivalent: 
\begin{compactenum}[(i)]
\item $G$ is an $FC$-group; 
\item the finitary-group coarse structure of $G$ is uniformly invariant; 
\item for every infinite cardinal $\kappa$ the $\kappa$-group coarse structure of $G$ is uniformly invariant.
\end{compactenum}  
\end{proposizione}

Thanks to Proposition \ref{prop:coarseuniformoperation} we can easily find a coarse group for which the multiplication is not bornologous. It is the aim of 
\begin{example}\label{ex:non_borno}
	Consider the free group $F_2$, generated by $\{a,b\}$.
	\begin{compactenum}[(i)]
		\item Let $\II=[\langle a\rangle]^{<\omega}$. Then, if we endow $F_2$ with $\E_{\II}^\lambda$, $\mu$ is not bornologous since $\bigcup\II$ is not normal.
		\item If we endow $F_2$ with the finitary-group coarse structure, $\mu$ is not bornologous by item (i) (or by Proposition\ref{criterion}, as $F_2$ is not $FC$, e.g., $\{a\}^G$ is not finite).
	\end{compactenum}
\end{example}

 The following example shows that the compact-group coarse structure of a locally compact group need not be uniformly invariant.

\begin{esempio} 
Fix a prime number $p$ and let $\theta\colon \Q_p \to \Q_p$ by the multiplication by $p$ in the $p$-adic numbers. Then $\theta$ is a topological automorphism of $\Q_p$. 
Let $G := \Q_p \rtimes \langle\theta \rangle$ by the semidirect product defined by means of action determined by this automorphism. 
Let $K=\Z_p$ be the compact group of $p$-adic integers. Then $K$ is a compact subgroup of $G$, yet $K^G$ coincides with $\Q_p$, so it is not relatively compact. Hence, $\rC(G)$ is not uniformly bilateral.
\end{esempio}

\section{Description of large scale properties by group ideals}\label{sec:group_ideals}

Group ideals are very useful to characterise some large-scale properties of spaces or maps. 

 The first example regards connected components.
 
   If $(G,\E_{\II})$  is a coarse group, then $\bigcup\II$ is a subgroup of $G$. One can associate a group ideal on $G$ to every subgroup $H\leq G$ in the following way: $\II_H=\{A\subseteq G\mid A\subseteq H\}$. The $\II_H$-group coarse structure is an example of what we will call {\em linear coarse structures}.
In particular, we see that for a coarse group, $\mathcal Q_G(e)=\bigcup\II$ is a subgroup of $G$, which is not normal in general.
In fact, we can pick a non-normal subgroup $H$ of a group $G$ and then $\bigcup\II_H=H$ is not normal (see, for another example, Example \ref{ex:non_borno}). Note that, in topological groups, the connected component of the identity is a normal subgroup. In particular, $(G,\E_\II)$ is connected if and only if $G=\bigcup\II$.

 Let $f,g\colon X\to(G,\E_{\II})$ be two maps from a set to a coarse group. Then $f$ and $g$ are close if and only if there exists $M\in\II$ such that, for every $x\in X$, $(f(x),g(x))\in E_M$ or, equivalently, $g(x)\in f(x)M$. In that situation, for the sake of simplicity, we write $f\sim_M g$. By choosing  $M= M ^{-1}$ symmetric, we can achieve to have $f\sim_M g$ precisely when $g\sim_M f$

One can obtain useful characterisations of morphisms in terms of group ideals as in Propositions \ref{prop:homomorphism_group_ideal} and \ref{prop:homo_unif_bounded_copres}.


\begin{proposizione}{\rm \cite{nicasros2}}\label{prop:homomorphism_group_ideal}
Let $G$ and $H$ be two groups and $f\colon G\to H$ be a homomorphism between them. Let $\II_G$ and $\II_H$ be two group ideals on $G$ and $H$ respectively. Then:
\begin{compactenum}[(i)]
\item $f\colon(G,\E_{\II_G})\to(H,\E_{\II_H})$ has uniformly bounded fibers if and only if $\ker f\in\II_G$;
\item $f\colon(G,\E_{\II_G})\to(H,\E_{\II_H})$ is bornologous if and only if $f(I)\in\II_H$, for every $I\in\II_G$;
\item $f\colon(G,\E_{\II_G})\to(H,\E_{\II_H})$ is effectively proper if and only if $f^{-1}(J)\in\II_G$, for every $J\in\II_H$ (i.e., $f$ is proper).
\end{compactenum}
\end{proposizione}

\begin{proposition}\label{prop:homo_unif_bounded_copres}
Let $f\colon (G,\E_{\II_G})\to(H,\E_{\II_H})$ be a homomorphism between coarse groups. Then the following properties are equivalent:
\begin{compactenum}[(i)]
 \item $f$ is uniformly bounded copreserving;
 \item $f$ is weakly uniformly bounded copreserving;
 \item for every $K\in\II_H$, there exists $L\in\II_G$ such that $K\cap f(G)\subseteq f(L)$.
\end{compactenum}
\end{proposition}

\begin{proof} The implication (i)$\to$(ii) follows from Proposition \ref{prop:eff_cop_wcop}.

(ii)$\to$(iii) Let $K\in\II_H$ and fix an element $L\in\II_G$ such that $E_K\cap(f(G)\times f(G))\subseteq (f\times f)(E_L)$. We claim that $K\cap f(G)\subseteq f(L)$. Let $k\in K\cap f(G)=(E_K\cap(f(G)\times f(G)))[e]$. Then $k\in ((f\times f)(E_L))[e]$, which means that there exists $(z,w)\in E_L$ such that $f(z)=e$ and $f(w)=k$. Since $w\in zL$ and $f(z^{-1}w)=k$, we have that $z^{-1}w\in L$, and so
$$
(e,k)=(f(e),f(z^{-1}w))\in (f\times f)(E_L)\quad\mbox{and}\quad k\in f(E_L[e])=f(L).
$$

(iii)$\to$(i) Let $E_K\in\E_{\II_H}$ be an entourage, where $K\in\II_H$, and $L\in\II_G$ that satisfies the hypothesis. We claim that, for every $g\in G$, $E_K[f(g)]\cap f(G)\subseteq f(E_L[g])$. Fix an element $g\in G$ and $k\in K$. Assume that $f(g)k\in E_K[f(g)]\cap f(G)$, which means that there exists $h\in G$ such that $f(g)k=f(h)$. Then $k\in f(G)$ and so there exists $l\in L$ such that $f(l)=k$. Finally, $f(g)k=f(g)f(l)=f(gl)\in f(E_L[g])$.
\end{proof}

The following corollary trivially follows from Propositions \ref{prop:homomorphism_group_ideal} and \ref{prop:homo_unif_bounded_copres} and Corollary \ref{coro:ls_bijective_ce}.

\begin{corollario}\label{coro:homo_ce}
Let $f\colon (G,\E_{\II_G})\to(H,\E_{\II_H})$ be a homomorphism between coarse groups. Then $f$ is a coarse equivalence if and only if the following conditions hold:
\begin{compactenum}[(i)]
	\item $\ker f\in\II_G$;
	\item $f$ is large-scale surjective;
	\item for every $K\in\II_G$, $f(K)\in\II_H$;
	\item for every $K\in\II_H$, there exists $L\in\II_G$ such that $K\cap f(G)\subseteq f(L)$.
\end{compactenum}
\end{corollario}

Group ideals are useful also to describe some categorical constructions, in particular products and quotients of coarse groups.

Let $\{G_i\}_{i\in I}$ be a family of groups, and $\E_i$ be a coarse structure on $G_i$, for every $i\in I$. For the sake of simplicity, we will denote by $\Pi_iE_i$ the entourage $\bigcap_i(p_i\times p_i)^{-1}(E_i)$, where $E_i\in\E_i$, for every $i\in I$. Note that, for every $(x_i)_i\in\Pi_iG_i$, $(\Pi_iE_i)[(x_i)_i]=\Pi_i(E_i[x_i])$, and thus, in particular, if, for every $i\in I$, $\E_i=\E_{\II_i}$ for some group ideal $\II_i$ and $E_i=E_{K_i}$, where $K_i\in\II_i$, 
\begin{equation}
\label{eq:prod_e}
(\Pi_iE_{K_i})[e]=\Pi_i(E_{K_i}[e])=\Pi_iK_i.
\end{equation}

\begin{proposizione}\label{prop:compatibilityproduct}
Let $\{(G_i,\E_{\II_i})$ be a family of coarse groups. Then the product coarse structure $\E$ on the direct product $\Pi_iG_i$ is a group coarse structure and it is generated by the base $\II=\{\Pi_iK_i\mid K_i\in\II_i,\,\forall i\in I\}$.
\end{proposizione}

\begin{proof} We want to use Proposition \ref{prop:compatible*}. Fix an element $E\in\E$, and, without loss of generality, suppose that $E=\Pi_iE_{K_i}$, where $K_i\in\II_i$, for every $i\in I$. It is easy to check that $(\Pi_iG_i)(\Pi_iE_{K_i})=\Pi_i(G_iE_{K_i})$. Since, for every $i\in I$, $\E_{\II_i}$ satisfies Proposition \ref{prop:compatible*}, $G_iE_{K_i}\in\E_{\II_i}$, and thus $(\Pi_iG_i)(\Pi_iE_{K_i})\in\E$. The fact that $\E=\E_{\II}$ follows from \eqref{eq:prod_e} and Proposition \ref{prop:compatible}(ii).
\end{proof}

Let us state a trivial, but useful, property.

\begin{fact}
If $f\colon G\to H$ is a homomorphism and $\II$ is a group ideal on $G$, then $f(\II)=\{f(K)\mid K\in\II\}$ is a group ideal on $H$. 
\end{fact}

\begin{proposition}\label{prop:quotient_ce}
Let $q\colon G\to G/N$ be a quotient homomorphism, and $\II$ be a group ideal on $G$. Then the map $q\colon(G,\E_{\II})\to(G/N,\E_{q(\II)})$ is bornologous and uniformly bounded copreserving. Moreover, the map $q$ is a coarse equivalence if and only if $N\in\II$.
\end{proposition}
\begin{proof} The first claim is trivial, thanks to Propositions \ref{prop:homomorphism_group_ideal}(i) and \ref{prop:homo_unif_bounded_copres}. If $q$ is a coarse equivalence, then Proposition \ref{prop:homomorphism_group_ideal}(ii) implies that $N=q^{-1}(e_{G/N})\in\II$. Let us focus on the opposite implication, which can be found also in \cite{nicasros2}. Suppose that $N\in\II$. Let $q(K)$, where $K\in\II$, be an arbitrary element of $q(\II)$. Then $q^{-1}(q(K))=K\ker q=KN\in\II$, which concludes the proof in virtue of Proposition \ref{prop:homomorphism_group_ideal}(ii) since $q$ is surjective.
\end{proof}

\begin{corollario}\label{coro:q_ce_compact}
Let $G$ be a topological group, and $K$ be a compact normal subgroup of $G$. Then the quotient map $q\colon G\to G/K$ is a coarse equivalence provided that both $G$ and $G/K$ are endowed with their compact-group coarse structures. 
\end{corollario}
\begin{proof}
Since $K$ is compact, the map $q$ is perfect and thus $q(\rC(G))=\rC(G/H)$. Hence, we can apply Proposition \ref{prop:quotient_ce} to conclude.
\end{proof}

As another corollary of Proposition \ref{prop:quotient_ce} we obtain: 

\begin{corollario}\label{quot}
If $G$ is a group with finite $G'$, then every group coarse structure $\E$ on $G$ is uniformly invariant. 
\end{corollario}

\begin{proof} In order to see that $\E_\II^\lambda = \E_\II^\rho$
consider the quotient map $q\colon G\to G/G'$ and equip $ G/G'$ with the quotient group coarse structure $\E_{f(\II)}^\lambda=\E_{f(\II)}^\rho$ (they coincide since 
the group ideal $f(\II)$ is uniformly bilateral in the abelian group $G/G'$). Since 
$$
q\colon (G, \E_\II^\lambda)\to (G/G', \E_{f(\II)}^\lambda)\quad \mbox{and}\quad q\colon(G, \E_\II^\rho)\to (G/G', \E_{f(\II)}^\rho)
$$ are coarse equivalences and  $\E_{f(\II)}^\lambda =\E_{f(\II)}^\rho$, we deduce that the identity $(G, \E_\II^\lambda)\to (G, \E_\II^\rho)$
is a coarse equivalence (actually, an asymorphism). 
\end{proof}

It is easy to see that a family $\mathcal U$ of subsets of a coarse group $(G,\E_{\II})$  is uniformly bounded if and only if there exists $K\in\II$ such that $U\subseteq E_K[x]=xK$, for every $U\in\mathcal U$ and $x\in U$ (it is a characterisation for coarse groups of the general definition given in Section \ref{sec:cs_cg}). If $S\in\II$, we say that $\mathcal U$ is {\em $S$-disjoint} if, for every pair of distinct elements $U,V\in\mathcal U$, $U\cap E_S[V]=U\cap VS=\emptyset$.

\begin{definizione}\label{def:asdim}
A coarse group $(G,\E_{\II})$ has asymptotic dimension at most $n$ ($\asdim(G,\E_{\II})\leq n$), where $n\in\N$, if, for every $S\in\II$, there exists a uniformly bounded cover $\mathcal U=\mathcal U_0\cup\cdots\cup\mathcal U_n$ such that, for every $i=0,\dots,n$, $\mathcal U_i$ is $S$-disjoint. The asymptotic dimension of $(G,\E_{\II})$ is $n$ if $\asdim(G,\E_{\II})\leq n$ and $\asdim(G,\E_{\II})>n-1$. Finally, $\asdim(G,\E_{\II})=\infty$ if, for every $n\in\N$, $\asdim(G,\E_{\II})>n$.
\end{definizione}
 Definition \ref{def:asdim} is the specification for coarse groups of the general definition of asymptotic dimension, which can be given for every coarse space (see \cite{roe}). Asymptotic dimension is the large-scale counterpart of Lebesgue covering dimension (see \cite{Eng}).

If we take a coarse group $(G,\E_{\II})$, then, for every $e\in K=K^{-1}\in\II$, $E_K\circ E_K=E_{K\cdot K}$. Hence, a coarse group is cellular if and only if, for every symmetric element $K\in\II$ containing the identity, $E_K^\square=E_{\langle K\rangle}\in\E_\II$, which means that $\langle K\rangle\in\II$. We have then showed that a coarse group $(G,\E_\II)$ is cellular if and only if $\II$ has a cofinal family, with respect of inclusion, consisting of subgroups. A group coarse structure satisfying that property is called {\em linear}. This concept will be investigated in the next section. The equivalence between cellular coarse groups and linear coarse groups was already pointed out in \cite{PetProt2}.

For some coarse groups we have the following criterion for cellularity, which is also proved in \cite{nicasros2}, but with a stronger hypothesis.
\begin{proposition}\label{prop:asdim_0}
Let $(G,\E_\II)$ be a coarse group such that there exists an element $K\in\II$ that algebraically generates the whole group $G$. Then $\asdim(G,\E_\II)=0$ if and only if $G\in\II$.
\end{proposition}
\begin{proof}
Without loss of generality, we can assume that $K=K^{-1}$. Since $G=\langle K\rangle$, for every $U\subsetneq G$, $U\subsetneq E_K[U]=UK$. Hence, the only possible $K$-disjoint cover $\mathcal U$ is $\mathcal U=\{G\}$. Finally, $\mathcal U$ is uniformly bounded if and only if $G\in\II$.
\end{proof}

\section{Linear coarse structures induced by cardinal invariants}\label{section_linear_coarse}

A topological abelian group $(G,\tau)$, and its topology $\tau$, are called \emph{linear} if $\tau$ has a local base at $0_G$  formed by open subgroups of $G$. In the non-abelian case some authors impose normality on the open subgroups forming the local base. Motivated by this folklore notion in the area of topological groups, we defined in the previous section the notion of a linear coarse structure. Explicitly, a group coarse structure $\E_\II$ on $G$ is linear if there exists a non-empty family $\BB$ of subgroups $H_i$ of $G$, such that $H_iH_j \in\BB$ and $\cl(\BB)=\II$ (note that $H_i\cup H_j\subseteq H_iH_j$).

Note that, if we want $\II$ to be connected, then we have to ask that $\II$ contains all finitely generated subgroups of $G$. 

As far as the group itself is not finitely generated (as a normal subgroup) linear coarse structures do not look trivial. For example, if $G$ is abelian, then for every uncountable cardinal $\kappa$ the $\kappa$-group coarse structure defined in Example \ref{ex:group_c_s}(iv) is linear. 
We use this example to introduce a more general construction, namely group coarse structures which come out from {\em cardinal} and {\em numerical invariants}.

\begin{definizione} 
	A {\em cardinal invariant}\index{cardinal invariant} $i(\cdot)$ for abelian groups is an assignment $G \mapsto i(G)$ of a cardinal number $i(G)$ to every abelian group $G$ in such a way that, if $G \cong H$, then $i(G) = i(H)$.  
	
Call a cardinal invariant $i$  
	\begin{compactenum}[$\bullet$]
	\item\emph{subadditive}, if $i(H_1 + H_2) \leq i(H_1)+ i(H_2)$ whenever $H_i$ ($i = 1,2$) are subgroups of some abelian group $G$; 
	\item \emph{additive}, if $i(G) = i(H) + i(G/H)$ whenever $H$ is a subgroup of $G$;
	\item\emph{monotone (with respect to quotients)}, whenever $i(G/H) \leq i(G)$ for any subgroup $H$ of $G$;
	\item\emph{monotone (with respect to subgroups)}, whenever $i(H) \leq i(G)$ for any subgroup $H$ of $G$;
	\item\emph{bounded}, whenever $i(G) \leq \lvert G\rvert$ for any group $G$;
	\item {\em continuous}, if $i$ is bounded and if $i(G) = \sup_{\lambda \in \Lambda}i(G_\lambda)$, when $G$ is a direct limit of 
its subgroups $(G_\lambda)_{\lambda \in \Lambda}$;
	\item\emph{normalised}, if $i(\{0\}) = 0$. 
	\end{compactenum}
\end{definizione} 

Obviously, additivity implies subadditivity
and monotonicity with respect to both quotients and subgroups. 

Sometimes it is convenient to consider numerical invariants instead of cardinal invariants. A {\em numerical invariant for abelian groups} is an assignment $G\mapsto j(G)\in\R_{\ge 0}\cup\{\infty\}$ such that $j(G)=j(H)$ provided that $G\cong H$. One can define boundedness, (sub)additivity, continuity, monotonicity, and normalisation also for numerical invariants in the same way. We say that $j$ is a \emph{length function}, if $j$ is continuous and additive.
Every cardinal invariant $i$ induces a numerical invariant $j_i$ by ``truncating from above" at $\omega$, i.e., by letting $j_i(G)=\min\{i(G),\infty\}$, for every abelian group $G$, where, for every $x\in\R_{\ge 0}$, $x<\infty$ and, for every infinite cardinal $\kappa$, we assume that $\infty\leq\kappa$.

\begin{esempio}
\begin{compactenum}[(i)]
 \item The {\em normalised cardinality}, defined by 
 $$
 \ell(G)=\begin{cases}\begin{aligned}&\lvert G\rvert, &\text{if $G$ is infinite,}\\
 & \log\lvert G\rvert, &\text{otherwise.}\end{aligned}\end{cases}
 $$ 
 This, maybe somewhat unusual, modification is due to the fact that the size $\lvert G\rvert$ is a cardinal invariant, but it fails to be normalised and subadditive (as far as finite groups are concerned).
 \item The free rank $r_0(G)$ and the $p$-ranks $r_p(G)$ of an abelian group $G$ are cardinal invariants. Hence also the rank $r(G)=\max\{r_0(G),\sup\{r_p(G)\mid p\in P\}\}$, where $P$ is the set of all prime numbers. In general, $r(G)\leq |G|$, they coincide when $r(G)$ is infinite. 
 \item Other invariants can be defined by using functorial subgroups. For example: 
  \begin{compactenum}[$\bullet$]
    \item (\cite{dgb}) the divisible weight:  $w_d(G) = \inf\{\lvert mG\rvert\mid m > 0\}$, 
    \item (\cite{dikshaA}) the divisible rank:  $r_d(G) = \inf\{r(mG)\mid m > 0\}$. 
 \end{compactenum}
 \item Using the idea from item (iii), for every cardinal invariant $i$ one can define its modification $i_d$ defined 
similarly to divisible 
 rank:  $i_d(G) = \inf\{i(mG)\mid m > 0\}$. It is bounded (normalised), whenever $i$ is, and it has particularly nice properties when $i$ is monotone with respect to taking subgroups and quotients. Then $i_d$ has the same properties and, moreover, $i_d$ is subadditive, whenever $i$ is.  
This shows that $r_d$ normalise, subadditive, bounded and monotone with respect to taking subgroups and quotients, while $w_d$ has all these properties beyond the first one. To obtain that one too one has to slightly modify its definition as follows 
$$
\widetilde w_d(G)  = \inf\{\ell(mG) \mid m > 0\}.
$$
It is easy to see that $\widetilde w_d(G) = w_d(G)$ is infinite for all unbounded groups, while $\widetilde w_d(G) = 0 < 1 =w_d(G)$
for all bounded groups. 
\end{compactenum}
All these cardinal invariants are subadditive and bounded, the normalised cardinality $\ell(\cdot)$,  the free rank $r_0$,  the divisible weight $w_d$ and the the divisible rank $r_k$ are also monotone with respect to quotients whereas $r$ and $r_p$ are not.
\end{esempio}

\subsection{The linear coarse structures associated to a cardinal invariant}\label{LinStr}

For a cardinal invariant $i$ we define now linear coarse structures depending on a fixed infinite cardinal $\kappa$.
To this end for any abelian group $G$ denote by $\BB_{i,\kappa}$ the family of all subgroups $H$ of $G$ with $i(H) < \kappa$. If $\kappa$ is infinite and $i$ bounded, $\BB_{i,\kappa}$ is non-empty. Here is a condition ensuring that $\BB_{i,\kappa}$ is a base of a group ideal.

\begin{claim}  Let $i$ be a normalised, subadditive cardinal invariant for abelian groups and let $\kappa$ be an infinite cardinal.  For every abelian group $G$ the family $\BB_{i,\kappa}$ is a base of a group ideal $\II_{i,\kappa}$ on $G$.
\end{claim}

\begin{proof}  If $H,K\in\BB_{i,\kappa}$, then $H\cup K\subseteq H+K\in\BB_{i,\kappa}$ since $i$ is subadditive. Moreover, for every subgroup $H$ of $G$ we have $-H=H$ and thus $H\in\BB_{i,\kappa}$ if and only if $-H\in\BB_{i,\kappa}$. 
\end{proof}

The following result is trivial.

\begin{proposizione}\label{classif1}
	Let $G$ be an abelian groups, $i$ be a normalised, subadditive cardinal invariant and $\kappa$ be an infinite cardinal. Then the trivial homomorphism $(G,\E_{\II_{i,\kappa}})\to\{0\}$ is a coarse equivalence if and only if $i(G)<\kappa$.
\end{proposizione}

For a fixed subadditive cardinal invariant $i$ and for any abelian group $G$ denote by $\BB_{i}^0$ the family of all subgroups $H$ of $G$ such that $i(H)=0$. If the cardinal invariant $i$ is subadditive and normalised, then the family $\BB_{i}^0$ is non-empty and defines a group ideal $\II_{i}^0$ inducing a cellular coarse structure on abelian groups. This construction can be carried out also in presence of a numerical invariant, and, moreover, for every cardinal invariant $i$, $\BB_{i}^0=\BB_{j_i}^0$.

\begin{proposizione}\label{prop:length_function}
	Let $G$ be a group and $j$ be a  normalised length function. Then group ideal $\II_{j}^0$ is generated by one element $I\subseteq G$. Moreover, the quotient map $q\colon G\to G/I$ is a coarse equivalence, provided that both groups are endowed with their $\II_j^0$-group coarse structures.
\end{proposizione}

\begin{proof} The subgroup $I:=\sum\{H\mid H\in\BB_{j}^0\}$ satisfies $\II_{j}^0=\cl(\{I\})$. The claim is trivial since $j$ is continuous and then $j(I)=\sup\{j(H)\mid H\in\BB_{j}^0\}=0$, which prove that $I\in\BB_{j}^0$.

The second statement follows from Proposition \ref{prop:quotient_ce} since $j(I)=0$.
\end{proof}

Let $G$ be an abelian group. Then $\II_{r_0}^0$ is a group ideal on it, since $r_0(\{0\})=0$. Moreover, since the numerical invariant induced by $r_0$ is a length function, we can apply Proposition \ref{prop:length_function} to prove that it is generated by the torsion subgroup $\Tor(G)$ of $G$. Moreover, $q\colon G\to G/\Tor(G)$ is a coarse equivalence. Note that $G/\Tor(G)$ is torsion-free.

The next issue we intend to face is ``how much'' the above group coarse structures can ``distinguish'' the groups, i.e., is there a great variety of groups that are not coarse equivalent with respect to the linear coarse structures just defined? 

\begin{proposizione}\label{classif2} Let $G$ and $H$ be two abelian groups, $i$ a cardinal invariant and $\kappa$ be an infinite cardinal. If there exists an homomorphism which is a coarse equivalence between $(G,\II_{i,\kappa})$ and $(H,\II_{i,\kappa})$ then either $i(G)<\kappa$ and $i(H)<\kappa$, or  $i(G)=i(H)$.
\end{proposizione}

Example \ref{mega:example}, with $i = r_0$ and $\kappa= \omega$, shows that the implication in the above proposition cannot be inverted. 

\subsection{Abelian groups with the functorial coarse structure $\E_{r_0,\kappa}$}\label{Sec:last}

In \cite{banhzar}, Banakh, Highes and Zarichnyi provided a complete characterisation of countable abelian groups endowed with their finitary-group coarse structures. Let us recall the following result, which was given with a slightly different, but equivalent statement.
\begin{theorem}\cite[Theorem 1]{banhzar}
For two countable abelian groups $G$ and $H$ endowed with their finitary-group coarse structure, the 
following three statements are equivalent: 
\begin{compactenum}[(i)]
\item $G$ and $H$ are coarsely equivalent;
\item $\asdim G =\asdim H$ and $ G$ and $H$ are both finitely generated or both infinitely generated;
\item $r_0(G)=r_0(H)$ and $G$ and $H$ are either both finitely generated or both infinitely generated.
\end{compactenum}
\end{theorem}
In the previous result, the free-rank played an important role. So it is reasonable to focus on the linear coarse structures associated to that cardinal invariant. In the sequel we fix the functorial coarse structure $\E_{r_0,\omega}$. 

\begin{osservazione}\label{deleted}
Let $G$ be a abelian group and let $H$ be a subgroup of $G$. Then the inclusion $j\colon H \to G$ is an asymorphic embedding.  
\end{osservazione}

Since $r_0(\Tor(G))=0$, Proposition \ref{prop:quotient_ce} implies that every abelian group $G$ is coarsely equivalent, via the quotient homomorphism $q\colon G\to G/\Tor(G)$, to a torsion-free abelian group. That's why we focus on torsion-free abelian groups in the sequel. Due to Remark \ref{deleted},
the study of the homomorphisms that are coarse equivalences can be reduced to the study of large subgroups. The next proposition provides a 
necessary condition for that. 

\begin{proposizione}\label{large} If a subgroup $H$ of a torsion-free abelian group $G$ is large, there exists $k \in \N$ such that 
\begin{equation}\label{last:Eq}
r_0(G/H) \leq k  \ \mbox{ and } \ r_p(G/H)  \leq k  \    \mbox{ for every prime }p. 
\end{equation}
\end{proposizione}

\begin{proof} Suppose that there exists a subgroup $S$  of $G$ with $H + S = G$ of finite free rank $k = r_0(S)$. Then $G/H \cong S/H\cap S$ is a quotient of a torsion-free group.  Therefore, $r_0(G/H) \leq k$ and all $p$-ranks $r_p(G/H) = r_p(S/H\cap S)$ are bounded by $k$. Indeed, while $r_0(G/H) \leq k$ obviously follows from the monotonicity of $r_0$, the latter inequality needs more care. 
As $k = r_0(S)$, we can assume without loss of generality that $S$ is a subgroup of $\Q^k$. Hence, $ S/H\cap S$ is a subgroup of $A:=\Q^k/H\cap S$. So it suffices to prove that 
\begin{equation}\label{last:Eq*}
r_p(A) \leq k. 
\end{equation}
By the definition of $r_p$, $r_p(A) = \dim_{\Z/p\Z} A[p]$, 
where $A[p] = \{a\in A\mid pa =0\}$. Let $S_1:= \{s\in \Q^k: ps \in H \cap S\}$. Then $A[p] \cong S_1/  H \cap S$. To prove that 
$\dim_{\Z/p\Z} S_1/  H \cap S\leq k$ pick a set $X$ with strictly more than $k$ elements of $S_1/  H \cap S $. To see that it is linearly dependent, 
consider a lifting $Y$ of $X$ in $S_1\leq \Q^k$ along the projection map $q\colon S_1 \to S_1/  H \cap S $. As $\lvert Y\rvert>k$, $Y$  satisfies a non-trivial relation $\sum_{y\in Y} k_yy = 0$ in $S_1$. If not all coefficients are divisible by $p$, the projection along $q$ immediately gives a linear dependence between the elements of $X$ in $S_1/  H \cap S$. If there exists some power $p^t$ dividing all $k_y$, then we can obtain a new linear combination $\sum_{y\in Y} \frac{k_y}{p^s}y = 0$, as $S_1$ is torsion-free. By choosing  the largest possible $t$, we obtain a  linear combination in which at least one coefficient is coprime with $p$, se we can argue as before. This proves (\ref{last:Eq*}). 
\end{proof}

In particular, if the inclusion $j\colon H \hookrightarrow G$ is a coarse equivalence, then (\ref{last:Eq}) holds for some $k\in \N$. 
We do not know whether this necessary condition implies that $j$ is a coarse equivalence in the case of arbitrary pairs $G$, $H$. 
Yet, we can say something in case the larger group $G$ is divisible. 

A group $G$ is {\em divisible} if, for every $y\in G$ and every $n\in\N\setminus\{0\}$, there exists $x\in G$ such that $x^n=y$. 
Every abelian group $G$ has a largest divisible subgroup $d(G)$.  Examples of divisible groups are $\Q$ and, for every prime $p$, the {\em Pr\" uffer $p$-group} $\mathbb{Z}_{p^{\infty}}$, i.e., the subgroup $\Z_{p^{\infty}}=\langle\{1/p^n\mid n\in\N\}\rangle\leq\T$. A group $G$ is called {\em reduced} if $d(G) =\{0\}$. Finite groups are reduced. 

Recall that a torsion-free group of the form 
\begin{equation}\label{last:Eq**}
G = \bigoplus_{i\in I}A_i, 
\end{equation}
where all $A_i$ are subgroups of $\Q$, is called {\em fully decomposable}. Free groups and divisible torsion-free groups are instances of  fully decomposable torsion-free groups. A fully decomposable group as in (\ref{last:Eq**}) is called {\em homogeneous}, if all groups $A_i$ are pairwise isomorphic. Note that  (\ref{last:Eq**}) is reduced precisely when all $A_i$ are proper subgroups of $\Q$.  

\begin{teorema}\label{ex:example} 
Let $D$ be a divisible group and (\ref{last:Eq**}) be a fully decomposable reduced subgroup of $D$. Suppose that one of the following conditions holds:
\begin{compactenum}[(i)]
\item $I$ is uncountable;
\item $G$ is homogeneous.
\end{compactenum}
Then the following properties are equivalent:
\begin{compactenum}[(a)]
\item the inclusion $j\colon G \hookrightarrow D$ is a coarse equivalence; 
\item there exists a homomorphism $f\colon G \to D$ that is a coarse equivalence;
\item $G$ and $D$ are bounded coarse spaces; 
\item $r_0(G)<\infty$ and $r_0(D)<\infty$. 
\end{compactenum}
\end{teorema}

\begin{proof} Under any of the two assumptions (i) or (ii), the implication (a)$\to$(b) is trivial, as well as the equivalence of (c) and (d), while (c) trivially implies (a) (actually, any homomorphism will do). It only remains to prove (b)$\to$(d)
when either (i) or (ii) holds.  
 The initial part of the argument coincides in both cases. 
  
Suppose that $f\colon G\to D$ is a homomorphism and a coarse equivalence. 
By Corollary \ref{coro:homo_ce}, 
$r_0(\ker f) < \omega$. Hence, $ K = \ker f$ is contained in a finite direct summand
$L:= \bigoplus_{i\in J}A_i$, with $J \subseteq I$, of $G$. Moreover, $f$ factorises through an injective homomorphism $f_0\colon G/K \to D$ and $G/K \cong L/K \oplus G_1$, where  $G_1:= \bigoplus_{i\in I\setminus J}A_i$. Since $r_0(L/K) \leq r_0(L)<\infty$, the projection $G/K \to G_1$
is a coarse equivalence. Therefore, the restriction $f_1\colon G_1 \to D$ is still an (injective) coarse equivalence, so $f(G) = f_1(G_1)$ must be a large subgroup of $G$. 
We may assume, from now on, that $G_1$ is simply a subgroup of $D$, identifying it with $f(G) = f_1(G_1)$. 
According to Proposition \ref{large} and (\ref{last:Eq}), there exists some $k\in \N$, such that 
\begin{equation}\label{last:Eq+}
r_0(D/G_1) \leq k  \ \mbox{ and } \ r_p(D/G_1)  \leq k  \    \mbox{ for every prime }p. 
\end{equation}
If $r_0(D) < \infty$, this implies $r_0(G_1) < \infty$
and consequently $r_0(G) < \infty$, hence we are done. Assume in the sequel that $r_0(D)$ is infinite. Hence, also $r_0(G_1)$ is infinite by 
(\ref{last:Eq+}). 

Since $D$ is divisible, the divisible hull $D_i=D(A_i) \cong \Q$ of each $A_i$, $i\in I\setminus J$, is contained in $D$ along with the direct sum $D':= D(G_1) =\bigoplus_{I\setminus J} D_i$. Hence, the quotient group $D/G_1$ contains a subgroup isomorphic to $D'/G_1 \cong \bigoplus_{I\setminus J} D_i/A_i$. Since $G$ is reduced, $A_i\ne D_i$ for every $i\in I$, so $D_i/A_i$ is a non-trivial torsion (divisible) group. Therefore, $r_{p_i}(D_i/A_i) >0$ for some prime $p_i$. 
There is some prime $q$, such that $p_i = q$ for infinitely many indexes $i\in I\setminus J$, so that $r_q(D'/G_1)$ is infinite. 
In the case (i) this is clear as $I$ is uncountable. In case (ii) this follows from the fact that all groups $D_i/A_i$ are pairwise isomorphic, torsion and non-trivial. This proves, that $r_q(D'/G_1)$ is infinite, hence $r_q(D/G_1)$ is infinite as well. This contradicts (\ref{last:Eq+}). \end{proof}

 With a slight modification the above proof we can give the following more precise result. Suppose that $f\colon G \to D$ is a homomorphism that is a coarse 
 equivalence and $G$ is fully decomposable, while $D$ is divisible. Then $r_0(G/d(G)) < \omega$ in case $G$ is either  uncountable  or homogeneous. In other words, if a  fully decomposable torsion-free abelian group $G$ is coarsely equivalent (i.e.,  ``as close as possible'' from the large-scale point of view) to a divisible group, then $G$ is also ``as close as possible'' to a divisible group from algebraic point of view.  

We are not aware if one can replace the group   (\ref{last:Eq**}) in the above theorem by an arbitrary reduced torsion-free group.

As a corollary we prove that there exists no homomorphism which is also a coarse equivalence between a divisible group and a free abelian group, in case at least one of them has infinite free-rank. 

\begin{corollario}\label{mega:example}
Let $D$ be a divisible torsion free abelian group of infinite free rank. Then: 
\begin{compactenum}[(i)]
  \item there is no homomorphism which is also a coarse equivalence from $D$ to any reduced abelian group $F$; 
  \item if $F$ is a free abelian group, then  there is no homomorphism from $F$ to $D$ which is also a coarse equivalence
\end{compactenum}		
 \end{corollario}

\begin{proof}  
(i) Assume the existence of a homomorphism $f\colon D\to F$ which is a coarse equivalence. Since $D$ is divisible and $F$ is reduced, $f$ is necessarily the null homomorphism. In particular, 
	the trivial homomorphism $G\to\{0\}$ must be  a coarse equivalence. By the above proposition, this yields $r_0(G)< \omega$, a contradiction.
	
(ii) Follows from Theorem \ref{ex:example}. 
\end{proof}
	
	Let us note that a much stronger result can be proved than just item (i) in the above theorem: if a homomorphism 
	$f\colon D\to G$ to a torsion-free group $G$ is a coarse equivalence, then $f(D)$ is a finite-co-rank subgroup of $G$. More precisely, 
	$G = f(D) \oplus G_1$, where $r_0(G_1)< \infty$ and $f(D) \cong D$ is divisible.

 In this subsection we have provided some results for coarse groups in which the notion of divisibility plays an important role. Let us also mention that divisibility has a great impact in some properties of the coarse structures on the subgroup lattices considered in \cite{DikProZav}.


\subsection{Small size vs small asymptotic dimension}\label{sub:Ban_Che_Lya}

For a coarse space $(X, \E)$ call a subset $A$ of $X$ {\em small} if for each large set $L$ of $X$ the set $L \setminus A$ remains large in $X$ (this notion, along with other similar notions for size, is due to \cite{ProBa}, see also \cite{ProtSurvey} for applications to groups, and \cite{DikZa_size} for further progress in this direction). 
Let $\Sm(X)$ denote the family of all small subsets of the coarse space $X$. Furthermore, let $\D_<(X)$ denote the family of all subsets $A$ with $\asdim A<\asdim X$. These two families are ideals in $X$.

Small sets are considered as the large-scale counterpart of nowhere dense subsets in topology (\cite{BanLya}). It is a classical result that in $\R^n$ the ideal of nowhere dense subsets coincides with the one of those subsets that have covering dimension strictly less than $n$. Banakh, Chervak and Lyaskovska showed the large-scale counterpart of this classical result, \cite[Theorem 1.6]{banchelya}, which states that, for every coarse space $X$, the inclusion $\D_<(X) \subseteq\Sm(X)$ holds, while the opposite inclusion holds if $X$ is coarsely equivalent to $\R^n$, endowed with its compact-group coarse structure. 

 Moreover, for locally compact abelian groups endowed with their compact-group coarse structure, the authors provide the following characterisation.


\begin{teorema}\cite[Theorem 1.7]{banchelya} For a locally compact abelian group the following properties are equivalent:
\begin{compactenum}[(i)]
\item $\D_<(G)=\Sm(G)$;
\item $G$ is compactly generated;
\item $G$ is coarsely equivalent to $\R^n$, for some $n\in\N$.
\end{compactenum}
\end{teorema}

They ask a description of the spaces $X$ when the equality $\D_<(X)=\Sm(X)$ holds true (\cite[Problem 1.3]{banchelya}). 
Obviously, it holds true when $G$ is compact, since then $\D_<(X) =  \Sm(X) = \{\emptyset\}$. 
Here we provide a wealth of counter-examples to this equality 
which are based on the following trivial observation. 
If, for a coarse space $X$, $\asdim X = 0$, then  $\D_<(X) = {\emptyset}$ consists of only the empty subset of $X$. Therefore, to provide examples where the equality $\D_<(X) =  \Sm(X)$ does not hold it suffices to find spaces $X$ with $\asdim X = 0$ and such that $X$ has a non-empty small set. 


%


\begin{proposizione}\label{prop:ban_che_lya}
Let $i$ be an subadditive, bounded cardinal invariant, $\kappa$ be an uncountable cardinal and $G$ be an abelian group with $i(G)\ge \kappa$.
Then $[G]^{<\kappa}\subseteq\Sm(G,\E_{\II_{i,\kappa}})$. In particular, 
$$
\D_<(G,\E_{\II_{i,\kappa}})=\{\emptyset\}\subsetneq[G]^{<\kappa} \subseteq\Sm(G,\E_{\II_{i,\kappa}}).
$$
\end{proposizione}

\begin{proof} Let $S$ be a subset of $G$ with $\lvert S\rvert< \kappa$. To check that $S\in \Sm(G,\E_{\II_{i,\kappa}})$
pick a large subset $A$ of $G$ and a subgroup $K\in\II_{i,\kappa}$ such that 
\begin{equation}\label{ver:last:eq}
A+K=G.
\end{equation}
 The 
subadditivity and boundedness of $i$, combined with (\ref{ver:last:eq}), entail 
$$
\kappa \leq i(G) \leq i(\langle A \rangle) + i(K).
$$ 
Along with $i(K) < \kappa$, this implies that $\kappa \leq i(\langle A \rangle)$. Therefore, boundedness of $i$ gives 
$$
|A|= |\langle A \rangle| \geq i(\langle A \rangle) \geq \kappa > |S|.
$$  
Hence, there exists an element $a\in A\setminus S$. 
The set $S-a:=\{s-a\mid s\in S\}$ belongs to $[G]^{<\kappa}$, so the subgroup $\langle S-a\rangle+K$ belongs to $\II_{i,\kappa}$.
On the other hand, it is easy to verify that $(A\setminus S)+(\langle S-a\rangle+K)=G$ (by the choice of $a$, $(A\setminus S)+(\langle S-a\rangle+K)$
contains $S$, hence contains $A$ as well, so (\ref{ver:last:eq}) applies). 
This proves that $A\setminus S$ is large, so $S$ is small. 

%
\end{proof}

We can refine Proposition \ref{prop:ban_che_lya} if we consider as cardinal invariant the normalised cardinality. In fact, it is not hard to prove the following statement: {\em Let $G$ be an infinite group with cardinality $\kappa$. Then $[G]^{<\kappa}\subseteq\Sm(G,\E_{\II_{\kappa}})$.}



\section{Quasi-homomorphisms}\label{section:quasi_hom}

\begin{definizione}\label{def:q_hom}[\cite{ros}]
A map $f\colon G\to (H,\E_\II)$ from a group $G$ to a coarse group $(H,\E_\II)$ is a {\em quasi-homomorphism} if the maps $f^\prime,f^{\prime\prime}\colon G\times G\to H$, where $f^\prime\colon(g,h)\mapsto f(gh)$ and $f^{\prime\prime}\colon(g,h)\mapsto f(g)f(h)$, are close (equivalently, if there exists $K\in\II$ such that, for every $g,h\in G$, $f(gh)\in f(g)f(h)K$). 

If $E\in\E_{\II}$ is a symmetric entourage such that $\{(f(gh),f(g)f(h))\mid g,h\in G\}\subseteq E$, then $f$ is called an {\em $E$-quasi-homomorphism}.
In case $E=E_M$ for some $M \in \II$, we briefly write $M$-quasi-homomorphism to say that $f(gh)\in f(g)f(h)M$, for every  $g,h\in G$. 
\end{definizione}

By taking $(H,\E_\II)=(\R,\E_{\BB_d})$, where $d$ is the usual euclidean metric on $\R$, we recover the classical notion.

Almost all the results of this section can be generalised for quasi-coarse structures on monoids or semi-coarse structures on loops (\cite{Za}). However, for the sake of simplicity and for the purpose of this paper, we prefer to state and prove them just for groups.

\begin{remark}\label{rem:computations_quasi_homo}
	Let $G$ be a group, $(H,\E_\II)$ be a coarse group, $M\in\II$, and $f\colon G\to H$ be an $M$-quasi-homomorphism. We can assume, without loss of generality, that $M=M^{-1}$. 
	Since $f(e_G)=f(e_G\cdot e_G)\in f(e_G)f(e_G)M$, we have that $e_H\in f(e_G)M = f(e_G)M^{-1}$ and $f(e_G)\in e_HM =M$. Moreover, for every $x\in G$, $f(e_G)=f(xx^{-1})\in f(x)f(x^{-1})M$,
	and thus, in particular,
	$f(x)^{-1}\in f(x^{-1})Mf(e_G)^{-1}\subseteq f(x^{-1})MM$.
	Thanks to this computation, in the sequel when we say that $f$ is an $M$-quasi-homomorphism, we assume that $M\in\II$ satisfies 
	$$
	f(e_G)\in M, \ \ 	f(y)^{-1}\in f(y^{-1})M\ \mbox{ and } \ f(y^{-1})\in f(y)^{-1}M,
	$$
	for every $y\in G$. 
\end{remark}

 Let us start with some very easy examples.
 
\begin{example}\label{ex:easy_quasi_homo}
Let $f\colon G\to(H,\E_{\II})$ be a map between a group and a coarse group.
\begin{compactenum}[(i)]
 \item If $f$ is a homomorphism, then $f$ is a quasi-homomorphism.
 \item $f$ is a quasi-homomorphism, if $f$ is {\em bounded} (i.e., $f(G)$ is bounded in $H$), or, equivalently, if $f(G)\in\II$  by Remark \ref{rem:computations_quasi_homo}.  
In particular,  $f$ is a quasi-homomorphism when $\II=\mathcal P(H)$. As a consequence, we have that every map $f\colon(G,\mathcal P(G))\to(H,\mathcal P(H))$ is both a quasi-homomorphism and a coarse equivalence.
 \item If $\II=\{e_H\}$, then $f\colon G\to H$ is a quasi-homomorphism if and only if it is a homomorphism.
 \item An asymorphism may not be a quasi-homomorphism. In fact, for example, for every group $G$, endowed with the discrete coarse structure $\E_{\{\{e\}\}}$, every bijective self-map $f\colon G\to G$ is automatically an asymorphism. However, $f$ is a quasi-homomorphism if and only if $f$ is 
an isomorphism, according to item (iii). Hence, a counterexample can be easily produced. 
\end{compactenum}
\end{example}

\begin{example} It is well-known that surjective homomorphisms preserve various properties of the domain, e.g., having finite rank. Let us see that the counterpart of this property remains true also for quasi-homomorphisms with respect to the group coarse structure induced by $r_0$
and $\kappa = \omega$ in the following weaker form. 

Let $f\colon G \to H$ be an $H_1$-quasi-homomorphism, where $H_1$ is a subgroup of $H$ with $r_0(H_1)<\infty$,  and suppose that $G$ is finitely generated. Then 
$$
r_0(\langle f(G)\rangle) < \infty.
$$ 
More precisely, if $H_1$ contains the images of all (finitely many) generators of $G$ (that can be achieved without loss of generality), then also $f(G)$ is contained in $H_1$, so has finite free rank.

Let $X$ be the finite set of generators of $G$ and assume that $f(X) \subseteq H_1$. We assume that $e\notin X$. 
We argue by induction on $n:=\lvert X\rvert$. The case $n = 0$, i.e., $G= \{0\}$, is trivial, so we may 
assume that $n>0$ and that the assertion is proved for $n-1$. Then $ X\ne \emptyset$ so we can fix
an element $x\in X$ and let $Y =  X \setminus \{x\}$ and $G_1 = \langle Y \rangle$. Then $f(G_1) \leq H_1$
by our inductive hypothesis. Take any $g \in G = G_1 + \langle x \rangle$, then $g = g_1 + kx$. 
Therefore,  $f(g) \in f(g_1) + f(kx) + H_1 =  f(kx) + H_1$, as $ f(g_1)\in H_1$. By our assumption, $f(x) \in H_1$. If $k>0$, then a simple inductive argument shows that $f(kx) \in H_1$ as well. 
If $k <0$, then $f(kx) = -f(-kx) + H_1$ (by Remark \ref{rem:computations_quasi_homo}). Now $-f(-kx) \in H_1$ and we are done. 
\end{example}

This example leaves open the question on whether quasi-homomorphisms preserve finiteness of rank: 

\begin{domanda}
If $f\colon G \to H$ is a quasi-homomorphism, with respect to the group coarse structure induced by $r_0$
and $\kappa = \omega$, and $r_0(G)< \infty$, is it true that $r_0(\langle f(G)\rangle) < \infty$ as well?
\end{domanda}

Here come two very important properties of quasi-homomorphisms.

\begin{proposition}\label{prop:close_quasi_homo}
Let $f,g\colon G\to(H,\E_\II)$ be two maps between a group $G$ and a coarse group $(H,\E_{\II})$. Suppose that $f\sim_M g$
for some $M \in \II$. If $M^H \in {\II}$,  then $f$ is a quasi-homomorphism if and only if $g$ is a quasi-homomorphism.
\end{proposition}

\begin{proof} Suppose that $K\in\II$ is an element such that $f$ is a $K$-quasi-homomorphism. Then, for every $x,y\in G$,
$$
g(xy) \in f(xy)M \subseteq f(x)f(y) KM\subseteq g(x)M g(y)MKM\subseteq g(x)g(y)M^HMKM,
$$
according to \eqref{eq:unif_bilateral}. Therefore, $g$ is a $M^H MKM$-quasi-homomorphism. 

The opposite implication can be similarly shown.
\end{proof}

Inspired by Proposition \ref{prop:close_quasi_homo}, the reader may think that every quasi-homomorphism is close to a homomorphism. However, this is not the case, as Example \ref{ex:no_hom_close}(i),(ii) shows.

\begin{example}\label{ex:no_hom_close}
\begin{compactenum}[(i)]
	\item Consider the floor map $\lfloor\cdot\rfloor\colon\R\to\Z$, which is a quasi-homomorphism if we endow $\Z$ with the finitary-group coarse structure. However, since $\R$ is a divisible group, the only homomorphism from $\R$ to $\Z$ is the null-homomorphism, which is not close to $\lfloor\cdot\rfloor$.
	\item Let $f\colon\Z\to 2\Z$ be the map that associates to every integer $n$ the largest even number smaller than $n$. If $2\Z$ is endowed with the finitary-group coarse structure, then $f$ is a quasi-homomorphism. However it is not close to any homomorphism.
	\end{compactenum}
\end{example}

\begin{proposition}\label{prop:comp_quasi-homo}
Let $G$ be a group, $(H,\E_{\II})$ and $(K,\E_{\mathcal J})$ be two coarse groups, $f\colon G\to H$ be a quasi-homomorphism, and $g\colon H\to K$ be a bornologous quasi-homomorphism. Then $g\circ f$ is a quasi-homomorphism.
\end{proposition}

\begin{proof} Suppose that $f$ is an $M$-quasi-homomorphism and $g$ is an $N$-quasi-homomorphism, for some $M\in\II$ and $N\in\mathcal J$. 
Then, for every $x,y\in G$,
$$
g(f(xy))\in g(f(x)f(y)M)\subseteq g(f(x))g(f(y)M)N\subseteq g(f(x))g(f(y))g(M)NN,
$$
where $g(M)NN\in\mathcal J$ (according to Proposition \ref{prop:borno_quasi-homomorphism}), and so $g\circ f$ is a $g(M)NN$-quasi-homomorphism.
\end{proof}

Note that, without the assumption of bornology of $g$ in Proposition \ref{prop:comp_quasi-homo}, it is not true that composition of quasi-homomorphisms is still a quasi-homomorphism (see Example \ref{ex:no_hom_close*}(i)). As mentioned in the introduction, this fact has prevented any categorical systematization of quasi-homomorphisms up to now.

\begin{example}\label{ex:no_hom_close*}
\begin{compactenum}[(i)]
	\item By using Example \ref{ex:easy_quasi_homo}(ii), we are able to construct two quasi-homomorphisms whose composite is not a quasi-homomorphism. Let $G$ be a group and $\II$ be a group ideal on it which is different from $\mathcal P(G)$. If $f\colon G\to(G,\E_{\II})$ is not a quasi-homomorphism, we have the following situation:
	$$
	G\xrightarrow{f}(G,\E_{\mathcal P(G)})\xrightarrow{id_G}(G,\E_{\II}),
	$$
	where both arrows are quasi-homomorphisms (the identity is a homomorphism), but their composite is not a quasi-homomorphism. For example, set $G=\Z$, $\II=[\Z]^{<\omega}$, and $f=\lvert\cdot\rvert$, the absolute value.
	\item The inverse of a bijective homomorphism is a homomorphism. However, it is not true a similar result for quasi-homomorphisms. Let $G$ be a group and $f\colon G\to G$ be a bijective map which is not a homomorphism. Then $f\colon(G,\E_{\{\{e\}\}})\to(G,\E_{\mathcal P(G)})$ is a quasi-homomorphism, while its inverse is not a quasi-homomorphism (using Example \ref{ex:easy_quasi_homo}(iii), $f^{-1}$ is a quasi-homomorphism if and only if it is a homomorphism, which is not true). In Corollary \ref{coro:c_inverse_surj} we give a condition that guarantees that we can revert a bijective quasi-homomorphism obtaining a quasi-homomorphism as well.	
\end{compactenum}
\end{example}

Also quasi-homomorphisms allow us to prove a result (Proposition \ref{prop:borno_quasi-homomorphism}) similar to Proposition \ref{prop:homomorphism_group_ideal}.

\begin{proposizione}\label{prop:borno_quasi-homomorphism}
Let $f\colon(G,\E_{\II_G})\to(H,\E_{\II_H})$ be a quasi-homomorphism between two coarse groups. Then
\begin{compactenum}[(i)]
\item $f$ is bornologous if and only if $f(I)\in\II_H$, for every $I\in\II_G$;
\item if $\II_H$ is uniformly bilateral, then $f$ is effectively proper if and only if $f$ is proper.
\end{compactenum}
\end{proposizione}
\begin{proof}
Both ``only if'' implications are trivial. Suppose that $f$ is an $M$-quasi-homomorphism for some $M=M^{-1}\in\II_H$.

(i,$\gets$) Let $K\in\II_G$ and take an arbitrary element $(x,xk)\in E_K$, where $x\in G$ and $k\in K$. Then
$$
f(xk)\in f(x)f(k)M\subseteq f(x)f(K)M,
$$
which implies that $(f(x),f(xk))\in E_{f(K)M}$. Thus $(f\times f)(E_K)\subseteq E_{f(K)M}\in\E_{\II_H}$ and so $f$ is bornologous.

(ii,$\gets$) Let $K\in\II_H$. Then, for every $(x,y)\in (f\times f)^{-1}(E_K)$, there exists $k\in K$ such that $f(y)=f(x)k$, and thus $f(x)^{-1}f(y)=k\in K$. We have
$$
f(x^{-1}y)\in f(x^{-1})f(y)M\subseteq f(x)^{-1}Mf(y)M\subseteq f(x)^{-1}f(y)M^{f(y)}M\subseteq KM^HM\in\II_H.
$$
Finally, $x^{-1}y\in f^{-1}(KM^HM)\in\II_G$ and $(x,y)=(x,xx^{-1}y)\in E_{f^{-1}(KM^HM)}$, which finishes the proof.
\end{proof}

\begin{theorem}\label{theo:quasi_homo_cinverse}
Let $f\colon(G,\E_G)\to(H,\E_H)$ be a quasi-homomorphism between coarse groups which is a coarse equivalence with coarse inverse $g\colon H\to G$. If $\E_H$ is uniformly invariant, then $g$ is a quasi-homomorphism.
\end{theorem}
\begin{proof}
Let $F\in\E_H$ be a symmetric entourage such that $f$ is an $F$-quasi-homomorphism. We claim that there exists $E\in\E_G$ such that, for every $x,y\in H$, $(g(xy),g(x)g(y))\in E$. Let $x,y\in H$. Then, since $\mu\colon H\times H\to H$ is bornologous (Proposition \ref{prop:coarseuniformoperation}),
$$
\begin{aligned}
(f(g(xy)),f(g(x)g(y)))&\,=(f(g(xy)),xy)\circ(xy,f(g(x))f(g(y)))\circ(f(g(x))f(g(y)),f(g(x)g(y)))\in\\
&\,\in S\circ(\mu\times\mu)(S,S)\circ F\in\E_H,
\end{aligned}
$$
where $S=\{(f(g(z)),z),(z,f(g(z)))\mid z\in H\}\in\E_H$, and thus it suffices to define $E:=(f\times f)^{-1}(S\circ(\mu\times\mu)(S,S)\circ F)$.
\end{proof}

\begin{theorem}\label{theo:unif_invariant_c_invariant}
Let $f\colon(G,\E_G)\to(H,\E_H)$ be a quasi-homomorphism between coarse groups which is a coarse equivalence. Then:
\begin{compactenum}[(i)]
\item if $\E_H$ is uniformly invariant, then $\E_G$ is uniformly invariant;
\item if a coarse inverse of $f$ is a quasi-homomorphism, then $\E_G$ is uniformly invariant if and only if $\E_H$ is uniformly invariant.
\end{compactenum}
\end{theorem}
\begin{proof}
Item (ii) follows from item (i). Assume that $\E_H$ is uniformly invariant and let $F\in\E_H$ be a symmetric entourage such that $f$ is a $F$-quasi-homomorphism. Let $E,E^\prime\in\E_G$. Then, for every $(x,y)\in E$ and $(z,w)\in E^\prime$,
$$
(f(xz),f(yw))=(f(xz),f(x)f(z))\circ(f(x)f(z),f(y)f(w))\circ(f(y)f(w),f(yw))\in E^{\prime\prime}\in\E_H,
$$
where $E^{\prime\prime}=F\circ(\mu((f\times f)(E),(f\times f)(E^\prime)))\circ F\in\E_N$, since $f$ is bornologous. Finally, since $f$ is effectively proper, $(\mu\times\mu)(E,E^\prime)\subseteq(f\times f)^{-1}(E^{\prime\prime})\in\E_G$, and thus Proposition \ref{prop:coarseuniformoperation} implies the claim.
\end{proof}

Note that the quasi-homomorphisms defined in Example \ref{ex:no_hom_close} are coarse inverses of the inclusions $i\colon\Z\to\R$ and $i\colon 2\Z\to\Z$, which are homomorphisms and coarse equivalences. Thus these two inclusions have no coarse inverses which are homomorphisms.

In Theorem \ref{theo:quasi_homo_cinverse}, the request of uniformly invariance of $\E_H$ is quite restrictive. In fact, we cannot apply the result to a coarse group $(H,\E_H)$ whose points have unbounded orbits under conjugacy. There is a trade-off between the uniformly invariance and the surjectivity of the map, as Corollary \ref{coro:c_inverse_surj} shows.

\begin{lemma}\label{lemma:coarse_inverse}
Let $(G,\E_G)$ and $(H,\E_H)$ be two coarse groups, $E\in\E_H$ be a symmetric entourage, $f\colon G\to H$ be a surjective $E$-quasi-homomorphism, and $s\colon H\to G$ be one of its sections (i.e., $f\circ s=id_H$). Suppose that $(f\times f)^{-1}(E)\in\E_G$. Then $s$ is a quasi-homomorphism.
\end{lemma}
\begin{proof}
Let $x,y\in H$. Since $f\circ s=id_H$,
$$
(f(s(xy)),f(s(x)s(y)))=(f(s(xy)),f(s(x))f(s(y)))\circ(f(s(x))f(s(y)),f(s(x)s(y)))\in\Delta_H\circ E=E,
$$
and thus $(s(xy),s(x)s(y))\in(f\times f)^{-1}(E)\in\E_G$.
\end{proof}

\begin{corollario}\label{coro:c_inverse_surj}
Let $f\colon(G,\E_G)\to(H,\E_H)$ be a surjective quasi-homomorphism, which is a coarse equivalence. Then there exists a coarse inverse of $f$ which is a quasi-homomorphism. In particular the inverse of a quasi-homomorphism which is an asymorphism, is a quasi-homomorphism.
\end{corollario}
\begin{proof}
Since $f$ is effectively proper, the conditions of Lemma \ref{lemma:coarse_inverse} are fulfilled and thus every section, which is a coarse inverse of $f$, is a quasi-homomorphism. The second statement trivially follows.
\end{proof}

\begin{osservazione}
Let $f\colon G\to(H,\E_\II)$ be a surjective $E_M$-quasi-homomorphism, for some $M\in\II$, between and abelian group $G$ and a coarse group $(H,\II)$. Then, for every $g,h\in G$,
\begin{equation}\label{eq:rem_surj_q_homo}
f(g)f(h)\in f(g+h)M=f(h+g)M\subseteq f(h) f(g)MM.
\end{equation}
In particular \eqref{eq:rem_surj_q_homo} shows that, for every $k,l\in H$, $[k,l]=k^{-1}l^{-1}kl\in MM$ and so the derived subgroup $[H,H]$ is contained in the subgroup $\langle M\rangle$ generated by $M$. If $\langle M\rangle\in\II$, then $H$ is coarsely equivalent to the abelian coarse group $(H/[H,H],\E_{q(\II)})$ since $q\colon H\to H/[H,H]$ is a coarse equivalence by Proposition \ref{prop:quotient_ce}.
\end{osservazione}

\section{Categories of coarse groups, functorial coarse structures and localisation}\label{sec:cat_coarse_grps}

The aim of this section is to discuss a categorical treatment of coarse groups. Recall that $\CS$ is the category of coarse spaces and bornologous maps between them. We now introduce a list of categories of coarse groups.
\begin{compactenum}[$\bullet$]
\item The category $\lCGQ$ ($\rCGQ$) has left coarse groups as objects (right coarse groups, respectively), and bornologous quasi-homomorphisms as morphisms.
\item The category $\CGQ$ is the intersection of $\lCGQ$ and $\rCGQ$, i.e., its objects are coarse groups whose coarse structures are uniformly invariant, and its morphisms are bornologous quasi-homomorphisms (according to Proposition \ref{prop:coarseuniformoperation}).
\item The category $\lCG$ ($\rCG$) has left coarse groups as objects (right coarse groups, respectively), and bornologous homomorphisms as morphisms.
\item The category $\CG$ is the intersection of $\lCG$ and $\rCG$, i.e., its objects are coarse groups whose coarse structures are uniformly invariant, and its morphisms are bornologous homomorphisms.
\item For any infinite cardinal $\kappa$, the subcategory $\kappaCGQ$ ($\kappaCG$, $\lkappaCG$, $\rkappaCG$) of $\CGQ$ (of $\CG$, $\lCG$, $\rCG$, respectively) whose objects are groups endowed with $\kappa$-group coarse structures.
\end{compactenum}
Thanks to Proposition \ref{prop:comp_quasi-homo}, composites of bornologous quasi-homomorphisms are still quasi-homomorphisms, and thus the categories whose morphisms are bornologous quasi-homomorphisms are indeed categories.

In diagram \eqref{forgetful}, we enlist the categories of coarse groups just defined, where the arrows represent forgetful functors. For the sake of simplicity, we do not include the categories $\rCGQ$, $\rCG$, and $\rkappaCG$.
\begin{equation}\label{forgetful}
\xymatrix{
	&\CS &\\
	&\lCGQ\ar[u] &\\
	\lCG\ar[ur] & &\CGQ\ar[ul]\\
	\lkappaCG\ar[u] &\CG\ar[ur]\ar[ul] &\kappaCGQ\ar[u]\\
	&\kappaCG\ar[ur]\ar[u]\ar[ul]&.
}
\end{equation}

Let $\XX$ be a category and $\sim$ be a {\em congruence} on $\XX$, i.e., for every $X,Y\in\XX$, $\sim$ is an equivalence relation in $\Mor_\XX(X,Y)$ such that, for every $f,g\in\Mor_\XX(X,Y)$ and $h,k\in\Mor_\XX(Y,Z)$, $h\circ f\sim k\circ g$, whenever $f\sim g$ and $h\sim k$. Hence the {\em quotient category $\XX/_\sim$} can be defined as the one whose objects are the same of $\XX$ and whose morphisms are equivalence classes of morphisms of $\XX$, i.e., $\Mor_{\XX/_{\sim}}(X,Y)=\{[f]_\sim\mid f\in\Mor_{\XX}(X,Y)\}$, for every $X,Y\in\XX/_\sim$. For example, the closeness relation $\sim$ is a congruence in $\CS$ and so the quotient category $\CSsim$ can be defined (\cite{DikZa_cat}). The isomorphisms of $\CSsim$ are precisely equivalence classes of coarse equivalences, whose inverses are equivalence classes of their coarse inverses. 

We will be interested in other quotient categories, namely $\CGQsim$, $\kappaCGQsim$, $\CGsim$, and $\kappaCGsim$, for every infinite cardinal $\kappa$ (see diagram \eqref{forgetful_sim}).
\begin{equation}\label{forgetful_sim}
\xymatrix{
	&\CSsim & \\
	&\CGQsim\ar[u] &\\
	\CGsim\ar[ur] & &\kappaCGQsim\ar[ul]\\
	&\kappaCGsim\ar[ul]\ar[ur] &
}
\end{equation}

Let us enlist some considerations on the previously defined categories, discussing the consequences of some results we proved in this setting.
\begin{remark}\label{rem:CGsim}
\begin{compactenum}[(i)]
\item The second assertion of Corollary \ref{coro:c_inverse_surj} implies that, if $X,Y\in\lCGQ$ ($X,Y\in\rCGQ$) and $f\colon X\to Y$ is a morphism in $\lCGQ$ ($\rCGQ$) such that $\Ufun f\colon\Ufun X\to\Ufun Y$ is an isomorphism of $\CS$, then $f$ is an isomorphism of $\lCGQ$ ($\rCGQ$, respectively).
\item Let $f\colon X\to Y$ be a morphism in $\CGQ$. Proposition \ref{prop:close_quasi_homo} implies that, for every other morphism $g\colon X\to Y$ in $\CS$ such that $g\sim f$, $g$ can be seen as a morphism of $\CGQ$. Thus the equivalence class of $f$ under closeness relation in $\CS$ is equal to the one in $\CGQ$.
\item Theorem \ref{theo:quasi_homo_cinverse} implies that, if $X,Y\in\CGQsim$ and $f\colon X\to Y$ is a morphism in $\CGQsim$ such that $\Ufun f\colon\Ufun X\to\Ufun Y$ is an isomorphism of $\CSsim$, then $f$ is an isomorphism of $\CGQsim$. Note that we cannot replace the category $\CGQsim$ with $\CGsim$, in fact there are homomorphisms which are coarse equivalences, but they have no coarse inverses which are homomorphisms (Example \ref{ex:no_hom_close}).
\end{compactenum}
\end{remark}

\subsection{Functorial coarse group structures}\label{sub:funct}
As announced in the Introduction, now that we have defined categories of coarse groups, we can introduce functorial coarse structures. A {\em functorial coarse structure} of groups is a concrete functor $\Ffun\colon{\bf Grp}\to\lCG$, where concrete means that 
$\Ufun\circ\Ffun$ is the identity functor, where $\Ufun\colon\lCG\to{\bf Grp}$ is the forgetful functor. A functorial coarse structure is called {\em \perf}, if for every morphism $f\colon G \to H$ in $\mathbf{Grp}$, the morphism $\Ffun(f)$ is uniformly bounded copreserving. In a similar (but appropriately modified) way 
we can define functorial coarse structures on topological groups, as functors $\Gfun\colon{\bf TopGrp}\to\lCG$.


\begin{osservazione}
Perfect functorial coarse structures $\Ffun\colon{\bf TopGrp}\to\lCG$ have another remarkable property, namely, for every surjective homomorphism $f$, $\Ffun(f)$ is a quotient in $\CS$ (and thus in $\lCG$). According to Propositions \ref{prop:homomorphism_group_ideal} and \ref{prop:homo_unif_bounded_copres}, an homomorphism $f\colon (G,\E_{\II_G})\to(H,\E_{\II_H})$ is both bornologous and uniformly bounded copreserving if and only if $f(\II_G)=\II_H\cap\mathcal P(f(G))$. Then, if $f$ is surjective, $f\colon(G,\E_{\II_G})\to(H,\E_{\II_H})\cong(G/\ker f,\E_{q(\II_G)})$ is a quotient also in the category $\CS$ (and thus in $\lCG$), as it is showed in \cite[Proposition 6.5]{DikZa_cat}.
\end{osservazione}

One can show that perfect functorial coarse structures on the category ${\bf Grp}$ of abstract abelian groups 
are completely determined by their ``values" on free groups $F_\kappa$ of $\kappa$ generators, where $\kappa$ is an arbitrary cardinal.  

\begin{proposition}\label{prop:funct_perf_on_free} 
Assume that a group ideal $\II_\kappa$ is assigned to each $F_\kappa$ in such a way that every 
homomorphism $f\colon F_\kappa \to F_\mu$ is bornologous and uniformly bounded copreserving when $\kappa$ and $\mu$ vary arbitrarily.  
Then this assignment can be extended to a perfect functorial coarse structure on the category $\mathbf{Grp}$
assigning to every group $G$ the group ideal $\II_G:= f(\II_\kappa)$, provided $q\colon F_\kappa \to G$ is a surjective homomorphism. 
\end{proposition}
\begin{proof}
(a Sketch of a proof) Use the properties of  $\II_\kappa$ in the hypotheses to show that:
\begin{compactenum}[(a)]
\item $\II_G$ is correctly defined (in particular, does not depend on the choice of $q$);
\item $G\mapsto (G, \II_G)$ is a perfect functorial coarse structure. 
\end{compactenum}
 It is enough to prove (a) since (b) will immediately follows. One can use the following two facts. First of all, every group is a quotient of some free group, so that every group can be endowed with a group ideal. Moreover, for every homomorphism $f\colon G \to H$ in $\mathbf{Grp}$ (including $id_G$) and for every pair of surjective homomorphisms $q\colon F_\kappa \to G$ and $q'\colon F_\mu \to H$ there is a lifting $\tilde f\colon F_\kappa \to F_\mu$ such that the following diagram commutes
$$
\begin{CD}
F_\kappa @>\tilde f>> F_\mu \\
@VV{q}V @VV{q^\prime}V \\
{G} @>f>> {H}.\\
\end{CD}
$$
\end{proof}

A similar result can be shown for the category ${\bf AbGrp}$ where the perfect functorial coarse structures are determined by their ``values'' on the free abelian groups $A_\kappa=\bigoplus_\kappa\Z$.

One can take as a useful application of Proposition \ref{prop:funct_perf_on_free} the case of functorial coarse structures on the class of all groups of size at most $\kappa$, where $\kappa$ is a fixed cardinal. In that case, every group $G$ with $\lvert G\rvert\leq\kappa$ is a quotient of the free group $F_\kappa$ and thus one can define the group ideals of the whole class from its group ideals $\II_\kappa$ 
that are ``invariant'' under endomorphisms of $F_\kappa$, i.e., such that, for every endomorphism $f$, $f\colon(F_\kappa,\E_{\II_{\kappa}})\to(F_\kappa,\E_{\II_{\kappa}})$ are bornologous.

\begin{proposition}
All the group-coarse structures defined in Example \ref{ex:group_c_s} but the metric-group coarse structures are functorial. Moreover, the discrete, the bounded and the $\kappa$-group coarse structures are \perf.
\end{proposition}
\begin{proof}
The proofs are trivial or follow from classical topological results. As for the left-coarse structure, we refer to \cite[Lemma 2.35]{ros}.
\end{proof}

In a forthcoming paper (\cite{DikZa_her}) we focus on a particular functorial coarse structure, namely the compact-group coarse structure. We will study the preservation of some properties (especially related to dimensions) along the Pontryagin functor and the Bohr functor.

\begin{teorema}\label{ThmFunct}
Let  $i$ be a normalised, subadditive cardinal invariant of abelian groups.
 Then the following properties are equivalent: 
	\begin{compactenum}[(i)]
		\item for every group $G$ and every subgroup $H\leq G$, either $i(G/H)\leq i(G)$ or $i(G/H) < \infty$ whenever $i(G)$ finite;
		\item for every infinite cardinal $\kappa$, $\E_{\II_{i,\kappa}}$  defines a cellular functorial coarse structure in the category of abelian group, i.e., every group homomorphism $f\colon G \to H$ is bornologous when $G$ and $H$ carry their linear coarse structures $\E_{\II_{i,\kappa}}$.
	\end{compactenum}
\end{teorema}
\begin{proof} 
	(i)$\to$(ii) Let $f\colon G\to H$ be a homomorphism between abelian groups. It is enough to notice that for each $K\subseteq G$, if $K\in\BB_{i,\kappa}$, we have $i(f(K))=i(K/\ker f)\le i(K)<\kappa$ provided $i(K)$ is infinite. If $i(K)$ is finite, then $i(f(K))=i(K/\ker f)$ is finite as well, so 
	$i(f(K))<\kappa$ again. Hence $f$ is bornologous thanks to Proposition \ref{prop:homomorphism_group_ideal}.
	
	(ii)$\to$(i) Let $G$ be an abelian group and $H$ be a subgroup. Let $\kappa$ be an infinite cardinal such that $i(G)<\kappa$. Since $f\colon(G,\E_{\II_{i,\kappa}})\to(G/H,\E_{\II_{i,\kappa}})$ is bornologous and $G\in\II_{i,\kappa}$, then $G/H\in\II_{i,\kappa}$, which means that $i(G/H)<\kappa$. Since the cardinal $\kappa$ can be taken arbitrarily, then $i(G/H)\leq i(G)$. To check the case when $G$ is finite, just take $\kappa = \omega$. 
\end{proof}

The property (i) of Theorem \ref{ThmFunct} is obviously implied by the fact that the cardinal invariant $i$ is monotone with respect to quotients.
Similarly to the proof of the implication (i)$\to$(ii) of Theorem \ref{ThmFunct} one obtains the proof of the following: 

\begin{proposizione}
Let $i$ be a normalised and subadditive cardinal invariant, monotone with respect to taking quotients. Then $\E_{\II_{i}^0}$ defines a functorial coarse structure.
\end{proposizione} 
%

If the cardinal invariant is the free rank or the normalised cardinality, then, for every infinite cardinal $\kappa$, $\E_{\II_{i,\kappa}}$ defines a perfect functorial coarse structure. In the general case we cannot find the precise conditions on $i$ that ensure this property: 
 
\begin{problem} Determine the properties of the cardinal invariant $i$ such that for every infinite cardinal the functorial coarse structure $\kappa$, $\E_{\II_{i,\kappa}}$  is perfect.
 \end{problem}

\subsection{Preservation of morphisms properties along pullbacks}\label{sub:pull}

Several categorical constructions in the category $\CS$ can be carried out in the categories of coarse groups. In particular, we focus here on pullbacks, which will be useful in the sequel. Since $\CS$ is a topological category (see \cite{DikZa_cat}), $\CS$ has, in particular, pullbacks. We can also give a precise description of the pullback of the diagram $Y\xrightarrow{f}Z\xleftarrow{g} X$ in $\CS$ as the triple $(P,u,v)$ in the following commutative diagram
\begin{equation}\label{diag:pull}
\begin{CD}
P @>u>> X \\
@V{v}VV @VV{g}V \\
Y @>f>> Z,\\
\end{CD}
\end{equation}
where $P:=\{(x,y)\in X\times Y\mid g(x)=f(y)\}$ is endowed with the coarse structure inherited by $X\times Y$, and $u$ and $v$ are the restrictions of the canonical projections. Note that, if the diagram $Y\xrightarrow{f}Z\xleftarrow{g} X$ is in $\lCG$, in $\CG$, in $\lkappaCG$, or in $\kappaCG$, then \eqref{diag:pull} belongs to $\lCG$, to $\CG$, to $\lkappaCG$, or to $\kappaCG$, respectively, and thus it is a pullback also in those categories.

\begin{proposition}\label{prop:pull_coarse_emb} 
The class $\mathcal V$ of all coarse embeddings in $\lCG$ is preserved along pullbacks in $\lCG$, i.e., if the diagram \eqref{diag:pull} is a pullback  where $g\in\mathcal V$, then also $v\in\mathcal V$.
\end{proposition}

\begin{proof}
Denote by $\II_Y$, $\II_X$, and $\II_P$ the group ideals associated to $Y$, $X$, and $P$, respectively. Proposition \ref{prop:compatibilityproduct} implies that $\II_P=(\II_X\times\II_Y)\cap\mathcal P(P)$. Thanks to Proposition \ref{prop:homomorphism_group_ideal}, it is enough to show that, for every $K\in\II_Y$, $v^{-1}(K)\in\II_P$. For every $(x,y)\in v^{-1}(K)$, $g(x)=f(y)$ and $y\in K$. Thus $x\in g^{-1}(f(K))$ and so, since $g$ is a coarse embedding and $f$ is bornologous,
$$
v^{-1}(K)\subseteq g^{-1}(f(K))\times K\in\II_X\times\II_Y,
$$
according to Proposition \ref{prop:homomorphism_group_ideal}.
\end{proof}
Let us now prove a variation of Proposition \ref{prop:pull_coarse_emb}.
\begin{corollario}\label{coro:pull_coarse_eq}
   The class $\mathcal V^\prime$ of all coarse equivalences in $\lkappaCG$ is preserved along pullbacks in $\lkappaCG$.
\end{corollario}
\begin{proof}
According to Proposition \ref{prop:pull_coarse_emb}, it is enough to show that if $g$ is large-scale surjective, then so it is $v$. First of all, it is easy to check that $v(P)=f^{-1}(g(X))$. Since $g$ is large-scale surjective, $\lvert Z:g(X)\rvert<\kappa$. Thus
$$
\lvert Y:v(P)\rvert=\lvert Y:f^{-1}(g(X))\rvert=\lvert f(Y):g(X)\cap f(Y)\rvert=\lvert Z:g(X)\rvert<\kappa.
$$
\end{proof}
We could have given a different proof of Corollary \ref{coro:pull_coarse_eq} without using Proposition \ref{prop:pull_coarse_emb}. In fact, since the $\kappa$-group coarse structure is functorial and \perf, according to Corollary \ref{coro:homo_ce}, it is enough to show that $\lvert\ker v\rvert<\kappa$ and $\lvert Y:v(P)\rvert<\kappa$.

\subsection{Localisation of a category, the case of $\kappaCGsim[\mathcal W^{-1}]$}\label{sub:loc}
The reader may be disappointed by Remark \ref{rem:CGsim}(iii). In fact, it would be desirable to have a category where all homomorphisms which are coarse equivalences are actually isomorphisms. The category $\CGQsim$ has that property, but is it the best choice? The aim of this subsection is to discuss (and give a precise meaning to) this question.

\begin{definizione}
Let $\XX$ be a category and $\mathcal W$ be a family of morphisms of $\XX$. A {\em localisation of $\XX$ by $\mathcal W$} (or {\em at $\mathcal W$}) is given by a category $\XX[\mathcal W^{-1}]$ and a functor $\Qfun\colon\XX\to\XX[\mathcal W^{-1}]$ such that:
\begin{compactenum}[(i)]
\item for every $w\in \mathcal W$, $\Qfun(w)$ is an isomorphism;
\item for any category $\YY$ and any functor $\Ffun\colon\XX\to\YY$ such that $\Ffun(w)$ is an isomorphism, for every $w\in \mathcal W$, there exists a functor $\Ffun_{\mathcal W}\colon\XX[\mathcal W^{-1}]\to\YY$ and a natural isomorphism between $\Ffun$ and $\Ffun_{\mathcal W}\circ\Qfun$;
\item for every category $\YY$, the map between functor categories $\cdot\circ\Qfun\colon\Funct(\XX[\mathcal W^{-1}],\YY)\to\Funct(\XX,\YY)$ is full and faithful.
\end{compactenum}
\end{definizione}
The localisation of a category by a family of morphisms, if it exists, it is unique.

Intuitively, if we localise a category $\XX$ by a family of morphisms $\mathcal W$, we enrich the family of morphisms of $\XX$ by imposing that the elements of $\mathcal W$ become isomorphisms. We would like to apply this idea to localise $\CGsim$ by the family $\mathcal W$ of all equivalence classes of homomorphisms which are coarse equivalences. 

\begin{domanda}\label{q:loc}
In the previous notations, does the localisation $\CGsim[\mathcal W^{-1}]$ exist? If yes, is it isomorphic to $\CGQsim$?
\end{domanda}

The  functor  $\Ufun\colon\CGsim\to\CGQsim$ takes every $w\in\mathcal W$ to an isomorphism $\Ufun(w)$. Hence, if $\CGsim[\mathcal W^{-1}]$ exists, and $\Qfun\colon\CGsim\to\CGsim[\mathcal W^{-1}]$ is the functor guaranteed by the definition, there exists a functor $\Ffun_{\mathcal W}\colon\CGsim[\mathcal W^{-1}]\to\CGQsim$ and a natural transformation between $\Ufun$ and $\Ffun_{\mathcal W}\circ\Qfun$.

The final part of this subsection will be devoted to construct the localisation of $\kappaCGQsim$, for every infinite cardinal $\kappa$, by the family $\mathcal W$ of all homomorphisms which are coarse equivalences.

The general definition of the localisation of a category is hard to use. However there are some special situations in which constructing it and working with it is easier.

\begin{definizione}[\cite{GabZis}]\label{def:calculus}
A pair $(\XX,\mathcal W)$ of a category $\XX$ and a class of morphisms $\mathcal W$ is said to admit a {\em calculus of right fractions} if the following conditions holds:
\begin{compactenum}[(i)]
\item $\mathcal W$ contains all identities and it is closed under composition;
\item ({\em right Ore condition}) given a morphism $w\colon X\to Z$ in $\mathcal W$ and any morphism $f\colon Y\to Z$ in $\XX$, there exist a morphism $w^\prime\colon T\to Y$ in $\mathcal W$ and a morphism $f^\prime\colon T\to X$ in $\XX$ such that the diagram
\[
\begin{CD}
T @>f^{\prime}>> X \\
@V{w^\prime}VV @VV{w}V \\
Y @>f>> Z\\
\end{CD}
\]
commutes;
\item ({\em right cancellability}) given an arrow $w\colon Y\to Z$ in $\mathcal W$ and a pair of morphisms $f,g\colon X\to Y$ such that $w\circ f=w\circ g$, there exists an arrow $w^\prime\colon T\to X$ in $\mathcal W$ such that $f\circ w^\prime=g\circ w^\prime$.
\end{compactenum}
The pair $(\XX,\mathcal W)$ is a {\em homotopical category} if, moreover, the following property is fulfilled:
\begin{compactenum}[(iv)]
\item ({\em $2$-out-of-$6$-property}) for every triple of composable morphisms $f\colon X\to Y$, $g\colon Y\to Z$ and $h\colon Z\to T$, if $g\circ f$ and $h\circ g$ are in $\mathcal W$, then so are $f$, $g$, $h$ (and, necessarily $h\circ g\circ f$).
\end{compactenum}
\end{definizione}

If $\XX$ is a category, a {\em span} (or {\em roof}, or {\em correspondence}) {\em from an object $X$ to an object $Y$} is a diagram of the form $X\xleftarrow{f}Z\xrightarrow{g}Y$, for some morphisms $f$ and $g$ of $\XX$. In this case, $f$ ($g$) is the {\em left leg} ({\em right leg}, respectively) of the span.

If $(\XX,\mathcal W)$ admits a calculus of right fractions, then we can construct $\XX[\mathcal W^{-1}]$ as follows. It has the same objects as $\XX$, while, as morphisms, we take the spans between objects of $\XX$ whose left legs belong to $\mathcal W$ under the following equivalence relation: to such spans
$$
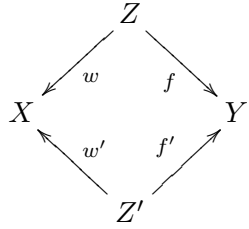
\xymatrix{
	& Z\ar[dl]^w\ar[dr]_f&\\
	X & & Y\\
	& Z^\prime\ar[ul]_{w^\prime}\ar[ur]^{f^\prime}&
}
$$
are equivalent if there exist an object $\overline Z$ and two morphisms $s\colon\overline Z\to Z$ and $t\colon\overline Z\to Z^\prime$ such that all the squares in 
$$
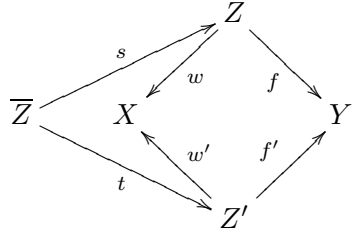
\xymatrix{
&	& Z\ar[dl]^w\ar[dr]_f&\\
\overline Z\ar[urr]^s\ar[drr]_t &	X & & Y\\
&	& Z^\prime\ar[ul]_{w^{\prime}}\ar[ur]^{f^{\prime}}&
}
$$
commute and $w\circ s=w^\prime\circ t\in\mathcal  W$.

In this category we define the composition of two morphisms as follows: if $X\xleftarrow{w} X^\prime\xrightarrow{f} Y$ and $Y\xleftarrow{w^\prime} Y^\prime\xrightarrow{g}Z$ are two representatives of their equivalence classes, because of Definition \ref{def:calculus}(ii), there exists another span $X^\prime\xleftarrow{w^{\prime\prime}}T\xrightarrow{h}Y^\prime$ such that all the  squares in
$$
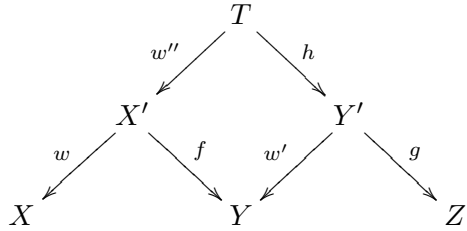
\xymatrix{
& & T\ar[dl]_{w^{\prime\prime}}\ar[dr]^h & &\\
& X^\prime\ar[dl]_{w}\ar[dr]^f & & Y^\prime\ar[dl]_{w^\prime}\ar[dr]^g &\\
X & & Y & & Z
}
$$
commutes, where $w^{\prime\prime}\in\mathcal  W$ and so does $w\circ w^{\prime\prime}$ (Definition \ref{def:calculus}(i)), and thus we can define the composite as the equivalence class of $X\xleftarrow{w\circ w^{\prime\prime}}T\xrightarrow{g\circ h}Z$.

The functor $\Qfun\colon\XX\to\XX[\mathcal W^{-1}]$ fix the objects and sends every morphism $f\colon X\to Y$ of $\XX$ in the span $X\xleftarrow{1_X}X\xrightarrow{f} Y$ (note, in fact, that $1_X\in \mathcal W$).

If we begin with a homotopical category, the functor $\Qfun$ is exact.

\begin{lemma}\label{lemma:CGsim_calculus}
Let $\mathcal W$ be the family of all equivalence classes of homomorphisms which are coarse equivalences. Then
\begin{compactenum}[(i)]
\item $\mathcal W$ contains all the identities and it is closed under composition;
\item $(\CGsim,\mathcal W)$ has the right cancellability property;
\item $(\CGsim,\mathcal W)$ has the $2$-out-of-$6$-property;
\item for every infinite cardinal $\kappa$, $(\kappaCGsim,\mathcal W^\prime)$, where $\mathcal W^\prime=\mathcal W\cap\kappaCGsim$, satisfies the right Ore condition.
\end{compactenum}
\end{lemma}
\begin{proof}
Item (i) is trivial.

(ii) Let $\Ufun$ be the forgetful functor from $\CGsim$ to $\CS$. Suppose that $w\colon Y\to Z$ belongs to $\mathcal W$ and $f,g\colon X\to Y$ is a pair of morphisms of $\CGsim$ such that $w\circ f=w\circ g$. Since $\Ufun(w)$ is an isomorphism, $f=g$ in $\CS$, and thus $f=g$ in $\CGsim$. Hence it is enough to put $w^\prime=1_X$.

Item (iii) can be proved similarly to item (ii), by using the functor $\Ufun$ and the fact that, for every $w\in\mathcal W$, $\Ufun(w)$ is an isomorphism of $\CS$.

(iv) Consider the diagram $Y\xrightarrow{f} Z\xleftarrow{w} X$ in $\kappaCG$, where $[w]\in\mathcal W^\prime$. Take the pullback 
\[
\begin{CD}
T @>f^{\prime}>> X \\
@V{w^\prime}VV @VV{w}V \\
Y @>f>> Z\\
\end{CD}
\]
in the category $\kappaCG$
. Then $w^\prime$ is a coarse equivalence, and $[w^\prime]\in\mathcal W^\prime$, according to Corollary \ref{coro:pull_coarse_eq}.
\end{proof}
 In Remark \ref{rem:lemma_3.5} we give a brief comment on the proof of Lemma \ref{lemma:CGsim_calculus}(iv).

From Lemma \ref{lemma:CGsim_calculus}, the following result immediately descends.
\begin{corollario}\label{coro:fin_loc}
For every infinite cardinal $\kappa$, the pair $(\kappaCGsim,\mathcal W)$, where $\mathcal W$ is the family of all equivalence classes of homomorphisms which are coarse equivalences, is a homotopical category and thus $(\kappaCGsim[\mathcal W^{-1}],\Qfun)$ exists and the functor $\Qfun\colon\kappaCGsim\to\kappaCGsim[\mathcal W^{-1}]$ is exact.
\end{corollario}

Let us specialise Question \ref{q:loc} in view of Corollary \ref{coro:fin_loc}, using the notation of Corollary \ref{coro:fin_loc}: 
\begin{domanda}
Is $\kappaCGsim[\mathcal W^{-1}]$ isomorphic to $\kappaCGQsim$?
\end{domanda}

\begin{remark}\label{rem:lemma_3.5}
According to Question \ref{q:loc}, we would like to know whether the localisation of the whole category $\CGsim$ by the family $\mathcal W$ of all homomorphisms which are coarse equivalences exists or not. One way to provide a positive answer is following the steps that led us to Corollary \ref{coro:fin_loc} and extending them to a more general setting. Then it is worth mentioning that Lemma \ref{lemma:CGsim_calculus}(i)--(iii) holds in general, while the only key point of the proof of Lemma \ref{lemma:CGsim_calculus}(iv) where we actually used the properties of the $\kappa$-group coarse structure is when we showed that $w^\prime$ has large image in $Y$. It is, in fact, the difference between Proposition \ref{prop:pull_coarse_emb} and Corollary \ref{coro:pull_coarse_eq}. If one could extend the proof of just that point, then Corollary \ref{coro:fin_loc} would be immediately generalised, providing a (maybe partial) answer to Question \ref{q:loc}. 
\end{remark}

\subsection*{Acknowledgments}
It is a pleasure to thank the referee for the careful reading and valuable suggestions.

\end{document}